\documentclass{amsart}

\usepackage{amsthm,amssymb,amsmath,hyperref}
\usepackage{graphicx}
\usepackage{color}
\usepackage{tikz}
\usepackage[utf8]{inputenc}

\theoremstyle{plain}
\newtheorem{thm}{Theorem}[section]

\newtheorem{cor}[thm]{Corollary}
\newtheorem{lem}[thm]{Lemma}
\newtheorem{prop}[thm]{Proposition}
\newtheorem{alg}[thm]{Algorithm}
\theoremstyle{definition}
\newtheorem{defn}[thm]{Definition}

\theoremstyle{remark}
\newtheorem{rem}[thm]{Remark}
\newtheorem{exa}[thm]{Example}
\theoremstyle{plain}

\numberwithin{equation}{section}

\newcommand{\del}{\delta}

\newcommand{\R}{{\mathbb R}}

\newcommand{\N}{{\mathbb N}}

\newcommand{\Z}{{\mathbb Z}}

\newcommand{\calA}{{\mathcal A}}

\newcommand{\calC}{{\mathcal C}}
\newcommand{\calD}{{\mathcal D}}

\newcommand{\calG}{{\mathcal G}}

\newcommand{\calL}{{\mathcal L}}

\newcommand{\calS}{{\mathcal S}}

\newcommand{\bfVa}{{\Vec{\textbf{a}}}}

\newcommand{\frakL}{{\mathfrak L}}

\def\udot#1{\ifmmode\oalign{$#1$\crcr\hidewidth.\hidewidth
    }\else\oalign{#1\crcr\hidewidth.\hidewidth}\fi}

\def\R{\mathbb{R}}
\def\Z{\mathbb{Z}}

\begin{document}
	
\title[]{On the general dyadic grids on $\R^d$}
\author{Theresa C. Anderson and Bingyang Hu}

\address{Theresa C. Anderson: Department of Mathematics, Purdue University, 150 N. University St., W. Lafayette, IN 47907, U.S.A.}%
\email{tcanderson@purdue.edu}

\address{Bingyang Hu: Department of Mathematics, University of Wisconsin, Madison, 480 Lincoln Dr., Madision, WI 53705, U.S.A.}%
\email{bhu32@wisc.edu}

\begin{abstract}
Adjacent dyadic systems are pivotal in analysis and related fields to study continuous objects via collections of dyadic ones.  In our prior work (joint with Jiang, Olson and Wei) we describe precise necessary and sufficient conditions for two dyadic systems on the real line to be adjacent.  Here we extend this work to all dimensions, which turns out to have many surprising difficulties due to the fact that $d+1$, not $2^d$, grids is the optimal number in an adjacent dyadic system in $\mathbb{R}^d$.  As a byproduct, we show that a collection of $d+1$ dyadic systems in $\R^d$ is adjacent if and only if the projection of any two of them onto any coordinate axis are adjacent on $\mathbb{R}$.  The underlying geometric structures that arise in this higher dimensional generalization are interesting objects themselves, ripe for future study; these lead us to a compact, geometric description of our main result.  We describe these structures, along with what adjacent dyadic (and $n$-adic, for any $n$) systems look like, from a variety of contexts, relating them to previous work, as well as illustrating a specific example.

\end{abstract}
\date{\today}

\thanks{}

\maketitle

\tableofcontents
\section{Introduction}
The purpose of this paper is to give an optimal description of adjacent dyadic systems (or more generally,  adjacent $n$-adic systems) in $\R^d$. Dyadic systems are ubiquitous in harmonic analysis, as well as many other fields.  Oftentimes, one wants to understand a continuous operator or object via its dyadic counterparts; our goal is to say, in an optimal and precise fashion, exactly what these dyadic counterparts are.   

The study of continuous objects via dyadic ones is a central theme in analysis and its application to many different areas of mathematics.  For instance, dyadic decompositions and partitions underlie the study of singular integral operators and maximal functions (among others), weight and function classes, partial differential equations, and number theory; our bibliography lists a few out of many references here.  In our recent paper \cite{AHJOW} joint with Jiang, Olson and Wei, we gave a necessary and sufficient condition on characterizing the adjacent $n$-adic systems on $\R$.  Here we generalize these results to higher dimensions.  Though we use ideas from \cite{AHJOW}, the construction of the analogous objects in $\R^d$ is not trivial; indeed we have to adapt our techniques from \cite{AHJOW} to a way that is compatible with the underlying lattice structure inherent in the construction of adjacent $n$-adic systems in $\R^d$. 

Let us begin with the definition of $n$-adic systems in $\R^d$, which is our main object of study. 

\begin{defn}
Given $n \in \N, n \ge 2$, a collection $\calG$ of left-closed and right-open cubes on $\R^d$ (that is, a collection of cubes in $\R^d$ of the form
$$
[a_1, a_1+\ell) \times \dots \times [a_d, a_d+\ell), \quad a_i \in \R, i=1, \dots, d, 
$$
where $\ell>0$ is the \emph{sidelength} of such a cube) is called \emph{a general dyadic grid with base $n$ (or $n$-adic grid)} if the following conditions are satisfied:
\begin{enumerate}
\item [(i).] For any $Q \in \calG$, its sidelength $\ell(Q)$ is of the form $n^k, k \in \Z$;
\item [(ii).] $Q \cap R \in \{Q, R, \emptyset\}$ for any $Q, R \in \calG$;
\item [(iii).] For each fixed $k\in \Z$, the cubes of a fixed sidelength $n^k$ form a partition of $\R^d$.
\end{enumerate}
In particular, if $n=2$, we also refer to such a collection \emph{a dyadic grid}, which is usually denoted by $\calD$. 
\end{defn}

The defining property of such a structure is a certain dyadic covering theorem.  The one that we use is due to Conde Alonso \cite{Conde}, and is optimal in terms of the number of grids required:

\begin{thm} {\cite[Theorem 1.1]{Conde}} \label{CondeThm}
There exists $d+1$ dyadic grids $\calD_1, \dots, \calD_{d+1}$ of $\R^d$ such that every Euclidean ball $B$ (or every cube) is contained in some cube $Q \in \bigcup\limits_{i=1}^{d+1} \calD_i$ satisfying that $\textrm{diam}(Q) \le C_d \textrm{diam}(B)$. The number of dyadic systems is optimal. 
\end{thm}

We make a remark that the optimal number $d+1$ in Theorem \ref{CondeThm} plays an important role throughout this paper.
Motivated by Theorem \ref{CondeThm}, we introduce the following definition of \emph{adjacent $n$-dic systems in $\R^d$}, our main object of study.

\begin{defn} \label{repdyadic}
Given $d+1$ many $\calG_1, \dots, \calG_{d+1}$ $n$-adic grids, we say they are \emph{adjacent} if for any cube $Q \subseteq \R^d$ (or any ball), there exists $i \in \{1, \dots, d+1\}$, and $R \in \calG_i$, such that
\begin{enumerate}
    \item [(1).] $Q \subseteq R$;
    \item [(2).] $\ell(R) \le C_{d, n}\ell(Q)$, where $C_{d, n}$ is a dimension constant that only depends on $d$ and $n$.
\end{enumerate}
\end{defn}

This characterizing property of adjacent dyadic systems is sometimes referred to as \emph{Mei's lemma} due to the work \cite{TM} on the torus (hence the definition we use is sometimes called the optimal Mei's lemma).  This property has been widely explored in a wide array of contexts and settings (see, \cite{PW}, \cite{LPW}, \cite{HK}), and has a long history; see the introduction of \cite{AHJOW} and also the monographs \cite{LN} and \cite{DCU} for details.  The applications of Mei's lemma are vast; adjacent dyadic systems are crucially used in the area of sparse domination (initiated by Lerner to prove the $A_2$ theorem in \cite{AL}, see also \cite{TH}, \cite{CR}, \cite{CDO} among others), functional analysis \cite{C2}, \cite{GJ}, \cite{LPW}, \cite{PW} and measure theory \cite{CP}.

Note that Conde Alonso's theorem only guarantees the existence of a collection of adjacent dyadic systems in $\R^d$; it does not say how to construct such systems in general nor how to tell if a system is adjacent.  Inspired by \cite{AHJOW}, we ask``\emph{what are the necessary and sufficient conditions so that a given collection of $d+1$ $n$-adic grids in $\R^d$ is adjacent?}"

In \cite{AHJOW}, we give a complete answer to this question on the real line, which we will briefly review in Section 2 below. In order to extend these results in \cite{AHJOW} to higher dimensions, we must deal with how $d+1$ $n$-adic grids, instead of only 2, interact with each other. The main idea to overcome such a difficulty is to work on a certain quantified version of the $n$-adic systems.  We introduced a one-dimensional analogue of this in \cite{AHJOW}, however, extending this concept to higher dimensions requires many new ideas.  The geometric structures that we define to quantify adjacency collapse into much simpler concepts on the real line, we provide perspective on this throughout the paper.


\subsection{Statement of the main result.} 

Suppose we are given $d+1$ $n$-adic grids $\calG_1, \dots, \calG_{d+1}$ in $\R^d$. Here is how to verify whether they are adjacent or not. 

\begin{alg}

\textit{Step I:} For each $i \in \{1, \dots, d+1\}$, write
$$
\calG_i:=\calG(\del_i, \calL_{\bfVa_i}),
$$
where 
\begin{enumerate}
    \item [(a).] $\del_i \in \R^d$ is called the \emph{initial position} of $\calG_i$;
    \item [(b).] $\calL_{\bfVa_i}: \N \to \N^d$ is the called \emph{the location function of $\calG_i$}, where $\bfVa_i \in \mathbb M_{d \times \infty} (\Z^n)$ is an infinite matrix with $d$ rows,  infinitely many columns, and entries belonging to $\Z^n$.
\end{enumerate}
Here, the term $\calG(\del_i, \calL_{\bfVa_i})$ is referred as the representation of the $n$-adic grid $\calG$ (see, Section 3 for more detailed information about this concept).

\medskip

\text{Step II:} 
Apply the following theorem, which is the main result of this paper.

\medskip

\begin{thm} \label{mainresult01}
Let $d, n, \del_i$ and $\bfVa_i$ be defined as above. Then the $n$-adic systems $\calG(\del_1, \calL_{\bfVa_1}), \dots, \calG(\del_{d+1}, \calL_{\bfVa_{d+1}})$ are adjacent if and only if the following conditions hold: 
\begin{enumerate}
\item [(1).] For any $\ell_1, \ell_2 \in \{1, \dots, d+1\}$ where $\ell_1 \neq \ell_2$, and $s \in \{1, \dots, d\}$, $\left(\del_{\ell_1} \right)_s-\left(\del_{\ell_2} \right)_s$ is $n$-far, that is, there exists some constant $C(\ell_1, \ell_2, s)>0$, such that for any $m \ge 0$ and $k \in \Z$, there holds
$$
\left| \left(\del_{\ell_1} \right)_s-\left(\del_{\ell_2} \right)_s- \frac{k}{n^m}\right| \ge \frac{C(\ell_1, \ell_2, s)}{n^m}; 
$$

\medskip

\item [(2).] For any $k_1, k_2 \in \{1, \dots, d+1\}, k_1 \neq k_2$, and $s \in \{1, \dots, d\}$, there holds
$$
    0< \liminf_{j \to \infty} \left|\frac{ \left[\calL_{\bfVa_{k_1}}(j)\right]_{s}-\left[\calL_{\bfVa_{k_2}}(j)\right]_{s}}{n^j}\right| \le \limsup_{j \to \infty} \left|\frac{ \left[\calL_{\bfVa_{k_1}}(j)\right]_{s}-\left[\calL_{\bfVa_{k_2}}(j)\right]_{s}}{n^j}\right|<1.
$$
\end{enumerate}
Here and in the sequel, we use $(\del)_s$ to denote the $s$-th component of a vector $\del \in \R^d$.
\end{thm}
\end{alg}

Note that Theorem \ref{mainresult01} is sharp, in the sense that the number of the dyadic systems is optimal. The proof of the above theorem uses the idea of representation of $n$-adic grids, which was introduced in \cite{AHJOW}. Moreover, combining with the one dimensional result (see, \cite[Theorem 3.8]{AHJOW} or Theorem \ref{dim1thm}),  Theorem \ref{mainresult01} is equivalent to the following result. 

\begin{thm} \label{maincor}
The collection of $n$-adic systems $\calG_1, \dots, \calG_{d+1}$ is adjacent if and only if for any $j \in \{1, \dots, d\}$ and $k_1, k_2 \in \{1, \dots d+1\}, k_1 \neq k_2$, $P_j(\calG_{k_1})$ and $P_j(\calG_{k_2})$ are adjacent on $\R$. 

Here, $P_j$ is the orthogonal projection onto the $j$-th axis of $\R^d$, and for any $n$-adic grid $\calG$, $P_j(\calG)$ is defined to be the collection of all $P_j(Q), Q \in \calG$. 
\end{thm}

\begin{rem}
Recall that in the classical approach of constructing an adjacent system in $\R^d$, what we usually do is first take any two adjacent dyadic systems (on $\R$) on each coordinate axis, and then take the Cartesian products of these dyadic systems. Note that this will give us a collection of $2^d$ adjacent dyadic systems in $\R^d$. 
We would like to point out that Corollary \ref{maincor} does not follow from this classical approach. First of all, our result is optimal, in the sense that the number of the $n$-adic systems is $d+1$, rather than $2^d$; moreover, our result provides a necessary and sufficient condition to tell whether a collection of $d+1$ $n$-adic grids are adjacent or not, rather than a single construction.
\end{rem}

Another interesting question to ask is whether there is a more inherent geometric approach to study the adjacency of the systems of the $n$-adic grids. More precisely, can we generalize the one dimensional result (see, Theorem \ref{dim1thm}) in a more parallel way, that respects the underlying geometric structure present in $d+1$ adjacent dyadic systems? 

In the second part of this paper, we give an affirmative and precise answer to the above question. The key idea is to  introduce the so-called \emph{fundamental structures of a collection of $d+1$ $n$-adic grids in $\R^d$}. These basic structures allow us to generalize the one dimensional result (see, Theorem \ref{dim1thm}) in a more heuristic way (see, Theorem \ref{mainresult}), whereas Theorem \ref{mainresult01} is much less obviously connected with the geometry of adjacent dyadic systems. The intuition for introducing these structures comes from a first, natural attempt to generalize the results in \cite{AHJOW} to $\R^d$ (see, Remark \ref{20200320rem01} and Section 8.1). Furthermore, all of these constructions are illustrated by a concrete example, which is elaborated on in detail before the proof of the main result (see, Theorem \ref{mainresult}).  This allows the reader to connect the underlying geometry with the results and examples in \cite{AHJOW} in a concrete way.

The novelty in this paper is that we generalize the results in \cite{AHJOW} via \emph{two different ways} that retain the key lattice structure implicit in the proof of \cite{Conde} for $d+1$ grids.  These generalizations are non-trivial, and motivate us to look at the underlying lattice structures inherent in the construction of $d+1$ grids and to expand them in a manner adaptable to the constructions underlying the main result (Theorem 3.8) in  \cite{AHJOW}.  These constructions allow us to better connect the geometry of the lattice with the arithmetic properties outlined in Theorem \ref{mainresult}, and likely will have applications to a variety of other problems in dyadic harmonic analysis.

The outline of this paper is as follows.  Part I begins with a brief reminder of our one dimensional results, followed by relevant definitions to state and prove our main theorem on necessary and sufficient conditions for adjacency -- this statement mirrors the one dimensional results only in notation, and does not shed light at the interesting geometric interactions taking place.  Therefore Part II is devoted to studying these.  Part II fully describes the rich geometry underlying the main result, including the fundamental structures which we define.  These descriptions not only motivate a restating of our main result that is geometrically driven, but provide a clear (and unifying) relationship between our one dimensional result and higher dimensions.  They also allow us to comment on the uniformity of such representations.  Finally, we illustrate everything with a concrete example, first introduced in Part I and revisited in Part II.\\






\part{Background and the proof of the main result.}

In the first part of this paper, we first make a short review of the one dimensional results, which were considered in \cite{AHJOW}. Then using the idea of representation of $n$-adic grids, we prove Theorem \ref{mainresult01}. Finally, we give an example on how to apply our main result.

\section{One dimensional results and some application}

Let us make a brief review of the case $d=1$, which was considered in our early work \cite{AHJOW}. The main question that was under the consideration in \cite{AHJOW} is the following ``\emph{Given two $n$-adic grids $\calG_1$ and $\calG_2$ on $\R$, what is the necessary and sufficient condition so that they are adjacent?}"

We start with recalling the following definition. 

\begin{defn} \label{20200411defn01}
A real number $\del$ is \emph{$n$-far} if there exists $C>0$ such that
\begin{equation}
\label{C delta}
  \left| \del-\frac{k}{n^m} \right| \ge \frac{C}{n^m}, \quad \forall m \ge 0, k \in \Z
\end{equation}
where $C$ may depend on $\del$ but independent of $m$ and $k$. 
\end{defn}

The key idea in \cite{AHJOW} to study this problem is to quantify each $n$-adic grids. More precisely, for any $n$-adic system $\calG$ on $\R$, we can find a number $\del \in \R$, and an infinite sequence $\textbf{a}:=\{a_0, a_1, \dots, a_j, \dots\} \in \{0, \dots, n-1\}^\infty$,  such that $\calG$ can be represented as $\calG(\del, \calL_{\textrm{\textbf{a}}})$, where  $\calL_{\textrm{\textbf{a}}}: \N \to \N$ is called the \emph{location function associated to \textbf{a}}, which is defined by 
$$
\calL_{\textrm{\textbf{a}}}(j):=\sum_{k=0}^{j-1} a_i n^k, \quad j \ge 1
$$
and $\calL_{\textrm{\textbf{a}}}(0)=0$.

Given two $n$-adic systems $\calG_1$ and $\calG_2$, let us write them as $\calG_1=\calG(\del_1, \calL_{\textrm{\textbf{a}}_1})$ and $\calG_2=\calG(\del_2, \calL_{\textrm{\textbf{a}}_2})$. Here is the main result in \cite{AHJOW}. 

\begin{thm} {\cite[Theorem 3.8]{AHJOW}} \label{dim1thm}
The $n$-adic grids $\calG(\del_1, \calL_{\textrm{\textbf{a}}_1})$ and $\calG(\del_2, \calL_{\textrm{\textbf{a}}_2})$ are adjacent if and only if
\begin{enumerate}
    \item [(1).] $\del_1-\del_2$ is $n$-far;
    \item [(2).]   There exists some $0<C_1 \le C_2<1$, such that
    $$
    0<C_1=\liminf_{j \to \infty} \left| \frac{\calL_{\textrm{\textbf{a}}_1}(j)-\calL_{\textrm{\textbf{a}}_2}(j)}{n^j} \right| \le \limsup_{j \to \infty} \left| \frac{\calL_{\textrm{\textbf{a}}_1}(j)-\calL_{\textrm{\textbf{a}}_2}(j)}{n^j} \right|=C_2<1. 
    $$
   
\end{enumerate}
\end{thm}

\begin{rem}
 To check whether $\del_1-\del_2$ is $n$-far or not, it suffices to check whether $T\left( \{\del_1-\del_2\} \right)$ is finite or not, where $\{\cdot\}$ indicate distance to the nearest integer, and for any $\del \in [0, 1)$, $T(\del)$ is defined to be the maximal length of consecutive $0$'s or $n-1$'s in the base $n$ representation of $\del$ (see, \cite[Theorem 2.8]{AHJOW}).
\end{rem}

 Although the representation of a $n$-adic grid is indeed not unique (see, the remark after \cite[Definition 3.11]{AHJOW} for the case $d=1$, or see, Proposition \ref{20200322lem01} for the general case), Theorem \ref{dim1thm} still enjoys some uniformness property.

 \begin{thm} {\cite[Theorem 3.14]{AHJOW}} \label{uniform}
 Under the same assumption of Theorem \ref{dim1thm},  let $\calG(\del'_1, \calL_{\textrm{\textbf{a}}'_1})$ and $\calG(\del'_2, \calL_{\textrm{\textbf{a}}'_2})$ be some other representations of $\calG_1$ and $\calG_2$, respectively. Then either
 $$
        \liminf_{j \to \infty} \left| \frac{\calL_{\textrm{\textbf{a}}'_1}(j)-\calL_{\textrm{\textbf{a}}'_2}(j)}{n^j} \right|=C_1 \quad \textrm{and} \quad \limsup_{j \to \infty} \left| \frac{\calL_{\textrm{\textbf{a}}'_1}(j)-\calL_{\textrm{\textbf{a}}'_2}(j)}{n^j} \right|=C_2
        $$
        or 
        $$
        \liminf_{j \to \infty} \left| \frac{\calL_{\textrm{\textbf{a}}'_1}(j)-\calL_{\textrm{\textbf{a}}'_2}(j)}{n^j} \right|=1-C_2 \quad \textrm{and} \quad \limsup_{j \to \infty} \left| \frac{\calL_{\textrm{\textbf{a}}'_1}(j)-\calL_{\textrm{\textbf{a}}'_2}(j)}{n^j} \right|=1-C_1.
        $$
 
\end{thm}

With the help of Theorem \ref{maincor}, we can easily generalize Theorem \ref{uniform} to higher dimensions.

\begin{cor} \label{20200412cor01}
Under the same assumption of Theorem \ref{mainresult01}, let $$
\calG(\del_1', \calL_{\bfVa_1'}), \dots , \calG(\del_{d+1}', \calL_{\bfVa_{d+1}'})$$ be some other representations of $\calG_1, \dots, \calG_{d+1}$, respectively. Moreover, for each $k_1, k_2 \in \{1, \dots, d+1\}$ and $s \in \{1, \dots, d\}$, denote
$$
D_1(k_1, k_2, s):=\liminf_{j \to \infty} \left|\frac{ \left[\calL_{\bfVa_{k_1}}(j)\right]_{s}-\left[\calL_{\bfVa_{k_2}}(j)\right]_{s}}{n^j}\right|, 
$$
$$
D_2(k_1, k_2, s):=\limsup_{j \to \infty} \left|\frac{ \left[\calL_{\bfVa_{k_1}}(j)\right]_{s}-\left[\calL_{\bfVa_{k_2}}(j)\right]_{s}}{n^j}\right|,
$$
and $D_1'(k_1, k_2, s), D_2'(k_1, k_2, s)$ similarly, then either
$$
D_1'(k_1, k_2, s)=D_1(k_1, k_2, s) \quad \textrm{and} \quad D_2'(k_1, k_2, s)=D_2(k_1, k_2, s)
$$
or
$$
D_1'(k_1, k_2, s)=1-D_2(k_1, k_2, s) \quad \textrm{and} \quad D_2'(k_1, k_2, s)=1-D_1(k_1, k_2, s).
$$
\end{cor}

\begin{proof}
Corollary \ref{20200412cor01} is an easy consequence of Theorem \ref{maincor} and Theorem \ref{mainresult01}, and we would like to leave the detail to the interested reader. 
\end{proof}

\section{Representation of $n$-adic grids}

Let us extend the concept of the representation of $n$-adic grids to higher dimension. The setting is as follows. 

\begin{enumerate}
    \item [(1).] $\del \in \R^d$, in particular, $\del$ should be thought as a vertex of some cube belonging to the $0$-th generation, and we may think it as the ``initial point" of our $n$-adic system;
    \item [(2).] An infinite matrix
    \begin{equation} \label{20200226eq02}
    \Vec{\textbf{a}}:=\left\{\Vec{a}_0, \dots, \Vec{a}_j, \dots \right\},
    \end{equation}
    where $\Vec{a}_j \in \{0,1,\dots , n-1\}^d, j \ge 1$;
    \item [(3).] The \emph{location function} associated to $\Vec{\textbf{a}}$: 
    $$
    \calL_{\bfVa}: \N \longmapsto \Z^d, 
    $$
    which is defined by 
    $$
\calL_{\bfVa}(j):=
\begin{cases}
\sum\limits_{i=0}^{j-1} n^i\Vec{a}_i, \hfill \quad \quad \quad j \ge 1;\\
\\ 
\Vec{0}, \hfill  \quad \quad \quad j=0. 
\end{cases}
$$
\end{enumerate}

For a vector $\del \in \R^d$, we use the notation $(\del)_i, 1 \le i \le d$ refers to the $i$-th component of $\del$.  Note that we will frequently be working with sets of $d$ vectors in $\R^d$, which we label $\del_1, \dots \del_d$.  Therefore the parentheses distinguish the selection from the components: $(\del_i)_s$ is the $s$-th component of the vector $\del_i$.

\begin{defn} \label{repsgrids}
Let $\del \in \R^d$, $\bfVa$ and $\calL_{\bfVa}$ be defined as above. Let $\calG(\del, \calL_{\bfVa})$ be the collection of the following cubes:
\begin{enumerate}
    \item [(1).] For $m \ge 0$, the $m$-th generation of $\calG(\del, \calL_{\bfVa})$ is defined as
\begin{eqnarray*}
&&\calG(\del)_m:=\calG(\del, \calL_{\bfVa})_m:=\Bigg\{ \left[ (\del)_1+\frac{k_1}{n^m}, (\del)_1+\frac{k_1+1}{n^m} \right) \times \dots \\
&& \quad \quad \quad  \quad \quad \quad \quad \quad \quad \quad  \quad \quad \quad \times \left[ (\del)_d+\frac{k_d}{n^m}, (\del)_d+\frac{k_d+1}{n^m} \right) \bigg | (k_1, \dots, k_d) \in \Z^d \Bigg\}.
\end{eqnarray*}

We make a remark sometimes we drop the dependence of the location function here, since location function only contributes to the negative generations;

\item [(2).] For $m<0$, the $m$-th generation is defined as 
\begin{eqnarray*}
&& \calG(\del, \calL_{\bfVa})_m:=\Bigg\{ \left[ (\del)_1+\left[ \calL_{\bfVa}(-m) \right]_1+\frac{k_1}{n^m}, (\del)_1+\left[ \calL_{\bfVa}(-m) \right]_1+\frac{k_1+1}{n^m} \right) \times \dots  \nonumber \\
&& \quad \quad  \times \left[ (\del)_d+\left[ \calL_{\bfVa}(-m) \right]_d+\frac{k_d}{n^m}, (\del)_d+\left[ \calL_{\bfVa}(-m) \right]_d+\frac{k_d+1}{n^m} \right)  \bigg | (k_1, \dots, k_d) \in \Z^d \Bigg\}. 
\end{eqnarray*}
\end{enumerate}

Or equivalently, 
\begin{enumerate}
    \item [(1).] For $m \ge 0$, the vertices of all $m$-th generation is defined as
    $$
    \calA(\del)_m:=\calA(\del, \calL_{\bfVa})_m:=\left\{ \del+\frac{\Vec{k}}{n^m}, \Vec{k} \in \Z^d \right\};
    $$
    
    \item [(2).] For $m<0$, the vertices of all $m$-th generation is defined as
    $$
    \calA(\del, \calL_{\bfVa})_m:=\left\{ \del+
    \calL_{\bfVa}(-m)+\frac{\Vec{k}}{n^m}, \Vec{k} \in \Z^d \right\};
    $$
\end{enumerate}

Note that in the above definition, the term $\del+\calL_{\bfVa}(-m)$ can be interpreted as the location of the ``initial point" (that is, $\del \in \calA(\del, \calL_{\bfVa})_0$) after choosing $n$-adic parents (with respect to the $0$-th generation) $(-m)$ times.  
\end{defn}

\begin{prop}
$\calG(\del, \calL_{\bfVa})$ is a $n$-adic grid in $\R^d$.
\end{prop}

\begin{proof}
If we restrict the grid to each axis, we obtain a $n$-adic grid with respect to that axis \cite{AHJOW}.  Since cubes are a one-parameter family, one can easily see (by contradiction) that cubes of a given level tile the space, two cubes are either contained one in the other or disjoint, each cube has $n^d$ children (each with $1/n^d$ of its parent's size) and exactly one parent.
\end{proof}

\begin{prop} \label{20200322lem01}
Given any $n$-adic grid $\calG$, we can find a $\del \in \R^d$ and an infinite matrix $\bfVa=\{\Vec{a}_0, \dots, \Vec{a}_j, \dots \}$, where $\Vec{a}_j \in \{0, \dots, d-1\}^d, j \ge 1$, such that
$$
\calG=\calG(\del, \calL_{\bfVa}). 
$$
However, this representation may not be unique. 
\end{prop}

\begin{proof}
The proof of this result is an easy modification of \cite[Proposition 4.10]{AHJOW}, and hence we leave the detail to the interested reader. While the fact that such a representation is not unique is also straightforward, one example in $\R^d$ would be
$$
\calG\left( (0, 0), \calL_{\bfVa_1} \right)=\calG\left(\left(2, 2 \right), \calL_{\bfVa_2}\right),
$$
where 
$$
\bfVa_1=\begin{pmatrix}
1 & 0 & 0 & 0 & \dots\\
1 & 0 & 0 & 0 & \dots
\end{pmatrix} \quad \textrm{and} \quad 
\bfVa_2=\begin{pmatrix}
0 & n-1 & n-1 & n-1 & \dots\\
0 & n-1 & n-1 & n-1 & \dots
\end{pmatrix}.
$$
\end{proof}

\section{Proof of the main result}

In this section, we prove the main result Theorem \ref{mainresult01}.

\subsection{Necessity} Suppose $\calG_1=\calG(\del_1, \calL_{\bfVa_1}), \dots, \calG_{d+1}=\calG(\del_{d+1}, \calL_{\bfVa_{d+1}})$ are adjacent. We prove the necessary part by contradiction. 

\medskip

Assume condition (a) fails, that is, there exists some $\ell_1, \ell_2 \in \{1, \dots, d+1\}$ with $\ell_1 \neq \ell_2$ and $s \in \{1, \dots, d\}$, such that for each $N_1 \ge 1$, there exists some $m_1 \ge 0$ and $K \in \Z$, such that
$$
\left| \left(\del_{\ell_1} \right)_s-\left(\del_{\ell_2} \right)_s- \frac{K}{n^{m_1}}\right|< \frac{1}{N_1 n^{m_1}}. 
$$
which implies the distance between the hyperplane $\left\{(x)_s=(\del_{\ell_1})_s\right\}$ and the hyperplane $\left\{(x)_s=(\del_{\ell_2})_s+K/n^{m_1}\right\}$ is less than $1/(N_1 n^{m_1})$. On the other hand, note that
$$
\left\{(x)_s=(\del_{\ell_1})_s\right\} \subset b\left[ \calG_{\ell_1, {m_1}} \right]
$$
and 
$$
\left\{(x)_s=(\del_{\ell_2})_s+\frac{K}{n^{m_1}}\right\} \subset b\left[ \calG_{\ell_2, {m_1}} \right],
$$
where $b\left[ \calG_{\ell_1, m_1} \right]$ is the union of all the boundaries of the cubes in $\calG_{\ell_1}$ with sidelength $1/n^{m_1}$, and similarly for the term $b\left[ \calG_{\ell_2, m_1} \right]$. 

Now the idea is to find two sufficiently closed points on the intersection of the boundaries of these $n$-adic grids. Without loss of generality, let us assume $s=\ell_1=1$ and $\ell_2=2$. Now let us consider the points
$$
p_1:=\{(x)_1=(\del_1)_1\} \cap \{(x)_2=(\del_3)_2 \} \cap \cdots \cap \{(x)_d=(\del_{d+1})_d \}
$$
and
$$
p_2:=\left\{(x)_1=(\del_2)_1+\frac{K}{n^{m_1}}\right\} \cap \{(x)_2=(\del_3)_2 \} \cap \cdots \cap \{(x)_d=(\del_{d+1})_d \} \},
$$
which satisfy the following properties:
\begin{enumerate}
    \item [(a).] $p_1 \in b\left[ \calG_{1, m_1} \right] \cap \bigcap\limits_{t=3}^{d+1} b\left[ \calG_{t, m_1} \right]$ and $p_2 \in \bigcap\limits_{t=2}^{d+1} b\left[ \calG_{t, m_1} \right]$; 
    \item [(b).] $\textrm{dist}(p_1, p_2)=\left| \left(\del_1 \right)_1-\left(\del_2 \right)_1- K/n^m\right|< 1/(N_1 n^{m_1})$. 
\end{enumerate}

Note that property (b) above allows us to choose an open cube $Q$ of radius $1/(N_1 n^{m_1})$ that containing both $p_1$ and $p_2$; while property (a) asserts that if there is a dyadic cube $D \in \calG_{\ell}, \ell \in \{1, \dots, d+1\}$ covering $Q$, then $\ell(D)>1/n^{m_1}$, and hence
$$
\ell(D)>N_1 \ell(Q). 
$$
This will contradict to the second condition in Definition \ref{repdyadic} if we choose $N_1$ sufficiently large (see, Figure \ref{Fig100}).

\begin{center}
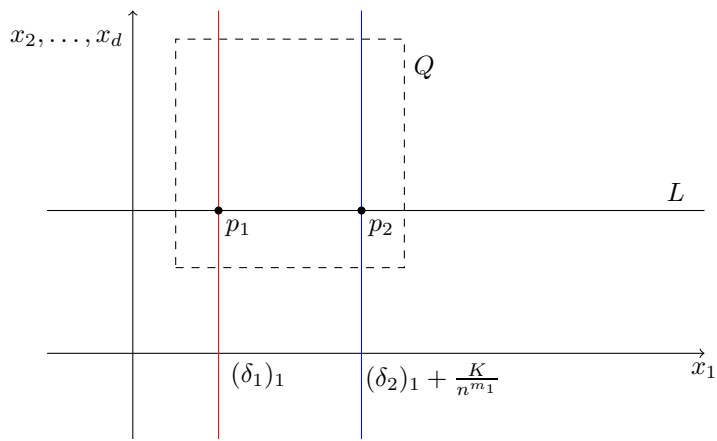
\begin{figure}[ht]
\begin{tikzpicture}[scale=3.8]
\draw (0, -.3) [->] -- (0,1.2);
\draw (-.3, 0) [->] -- (2,0);
\fill (2, 0) node [below] {$x_1$};
\fill (0, 1.1) node [left] {$x_2, \dots, x_d$}; 
\draw [opacity=1, red]  (.3, -.3)--(.3, 1.2);
\fill (.44, 0) node [below] {$(\del_1)_1$};
\draw [opacity=1, blue]  (.8, -.3)--(.8, 1.2);
\fill (1.05, 0) node [below] {$(\del_2)_1+\frac{K}{n^{m_1}}$};
\draw [opacity=1] (-.3, .5)--(2, .5); 
\fill  (.3, .5) circle [radius=.4pt];
\fill  (.8, .5) circle [radius=.4pt];
\fill (.37, .5) node [below] {$p_1$};
\fill (.87, .5) node [below] {$p_2$};
\fill (1.9, .5) node [above] {$L$};
\draw [dashed] (.15, .3)--(.95, .3)--(.95, 1.1)--(.15, 1.1)--(.15, .3); \fill (.95, 1) node [right] {$Q$};
\end{tikzpicture}
\caption{$p_1$, $p_2$, the hyperplane $\{(x)_1=(\del_1)_1\}$ (red part), the hyperplane $\left\{(x)_1=(\del_2)_1+K/n^{m_1}\right\}$ (blue part), the line $L:=\{(x)_2=(\del_3)_2 \} \cap \cdots \cap \{(x)_d=(\del_{d+1})_d \}$, and the cube $Q$ containing both $p_1$ and $p_2$.}
\label{Fig100}
\end{figure}
\end{center}

\medskip

Next, expecting a contradiction again, we assume (b) fails. The proof for this part is very similar to the previous one. Let us consider two different cases. 

\medskip

\textit{Case I:} There exists some $k_1, k_2 \in \{1, \dots, d+1\}, k_1 \neq k_2$ and $s \in \{1, \dots, d\}$, such that
\begin{equation} \label{20200412eq01}
\liminf_{j \to \infty} \left|\frac{ \left[\calL_{\bfVa_{k_1}}(j)\right]_{s}-\left[\calL_{\bfVa_{k_2}}(j)\right]_{s}}{n^j}\right|=0.
\end{equation}
Again, for simplicity, we may assume $s=k_1=1$ and $k_2=2$. By \eqref{20200412eq01}, for any $\varepsilon>0$, there exists some $j_1$ sufficiently large, such that
$$
 \left| \left[\calL_{\bfVa_1}(j_1)\right]_1-\left[\calL_{\bfVa_2}(j_1)\right]_1 \right|<\varepsilon \cdot n^{j_1},
$$
which implies
$$
 \left|\left[\del_1+\calL_{\bfVa_1}(j_1)\right]_1-\left[\del_2+\calL_{\bfVa_2}(j_1)\right]_1 \right|<2\varepsilon \cdot n^{j_1},
$$
since we assume $j_1$ is sufficiently large. Now we can exactly follow the idea in part (a) now. More precisely, we define
$$
q_1:=\{(x)_1=\left[\del_1+\calL_{\bfVa_1}(j_1)\right]_1\} \cap \bigcup_{t=2}^d \left\{(x)_t=\left[\del_{t+1}+\calL_{\bfVa_{t+1}}(j_1) \right]_t \right\}  
$$
and 
$$
q_2:= \bigcup_{t=1}^d \left\{(x)_t=\left[\del_{t+1}+\calL_{\bfVa_{t+1}}(j_1) \right]_t \right\},
$$
which enjoy similar properties as $p_1$ and $p_2$:
\begin{enumerate}
    \item [(a).] $q_1 \in b\left[ \calG_{1, -j_1} \right] \cap \bigcap\limits_{k=3}^{d+1} b\left[ \calG_{k, -j_1} \right]$ and $q_2 \in \bigcap\limits_{k=2}^{d+1} b\left[ \calG_{k, -j_1} \right]$;
    \item [(b).] $\textrm{dist}(q_1, q_2)<2 \varepsilon \cdot n^{j_1}$.
\end{enumerate}
Then desired contradiction will follow by taking an open cube with sidelength $2\varepsilon \cdot n^{j_1}$ containing both $q_1$ and $q_2$, where $\varepsilon$ is sufficiently small.

\medskip

\textit{Case II:}  There exists some $k_1, k_2 \in \{1, \dots, d+1\}, k_1 \neq k_2$ and $s \in \{1, \dots, d\}$, such that
\begin{equation} \label{20200412eq02}
\limsup_{j \to \infty} \left|\frac{ \left[\calL_{\bfVa_{k_1}}(j)\right]_{s}-\left[\calL_{\bfVa_{k_2}}(j)\right]_{s}}{n^j}\right|=1. 
\end{equation}
The proof for the second case is an easy modification of the first one. Indeed, \eqref{20200412eq02} implies that for any $\varepsilon>0$, there exists some $j_2$ sufficiently large, such that either
$$
\left| \left[\del_{k_1}+\calL_{\bfVa_{k_1}}(j_2)\right]_{s}-\left[\del_{k_2}+\calL_{\bfVa_{k_2}}(j_2)+n^j \vec{e}_s\right]_{s} \right|<\varepsilon \cdot n^{j_2},
$$
or
$$
\left| \left[\del_{k_1}+\calL_{\bfVa_{k_1}}(j_2)\right]_{s}-\left[\del_{k_2}+\calL_{\bfVa_{k_2}}(j_2)-n^j \vec{e}_s\right]_{s} \right|<\varepsilon \cdot n^{j_2},
$$
holds, where in the above estimates, $\vec{e}_s$ refers to the stand unit vector in $\R^d$ with $s$-th entry being $1$. The rest of the proof is the same as Case I. 

\medskip

\subsection{Sufficiency}

Suppose the conditions (a) and (b) hold, that is,
\begin{enumerate}
    \item [(1).] For any $\ell_1, \ell_2 \in \{1, \dots, d+1\}$ where $\ell_1 \neq \ell_2$, and $s \in \{1, \dots, d\}$,  there exists some constant $C(\ell_1, \ell_2, s)>0$, such that for any $m \ge 0$ and $k \in \Z$, there holds
\begin{equation} \label{20200412eq101}
\left| \left(\del_{\ell_1} \right)_s-\left(\del_{\ell_2} \right)_s- \frac{k}{n^m}\right| \ge \frac{C(\ell_1, \ell_2, s)}{n^m}; 
\end{equation}

\medskip

\item [(2).] For any $k_1, k_2 \in \{1, \dots, d+1\}, k_1 \neq k_2$, and $s \in \{1, \dots, d\}$, there holds
\begin{equation} \label{20200412eq102}
    0< \liminf_{j \to \infty} \left|\frac{ \left[\calL_{\bfVa_{k_1}}(j)\right]_{s}-\left[\calL_{\bfVa_{k_2}}(j)\right]_{s}}{n^j}\right| \le \limsup_{j \to \infty} \left|\frac{ \left[\calL_{\bfVa_{k_1}}(j)\right]_{s}-\left[\calL_{\bfVa_{k_2}}(j)\right]_{s}}{n^j}\right|<1.
\end{equation}
\end{enumerate}
Denote 
$$
C_1:=\min_{1 \le \ell_1 \neq  \ell_2 \le d+1, 1 \le s \le d} C(\ell_1, \ell_2, s), \quad D_1:=\min_{1 \le k_1 \neq k_2 \le d+1, 1 \le s \le d} D_1(k_1, k_2, s), 
$$
and
$$
D_2:=1-\max_{1 \le k_1 \neq k_2 \le d+1, 1 \le s \le d} D_2(k_1, k_2, s),
$$
where the quantities $D_1(k_1, k_2, s)$ and $D_2(k_1, k_2, s)$ are defined in Corollary \ref{20200412cor01}. Note that $C_1, D_1, D_2>0$ by our assumption \eqref{20200412eq102}. Moreover, by \eqref{20200412eq102}, there exists some $N \in \N$, such that for any  $k_1, k_2 \in \{1, \dots, d+1\}, k_1 \neq k_2$, $s \in \{1, \dots, d\}$ and $j \ge N$, there holds
\begin{equation} \label{20200412eq103}
\frac{D_1}{2}< \left|\frac{ \left[\calL_{\bfVa_{k_1}}(j)\right]_{s}-\left[\calL_{\bfVa_{k_2}}(j)\right]_{s}}{n^j}\right|<1-\frac{D_2}{2}.
\end{equation}

Recall the goal is to show that the collection $\calG_1=\calG(\del_1, \calL_{\bfVa_1}), \dots, \calG_{d+1}=\calG(\del_{d+1}, \calL_{\bfVa_{d+1}})$ is adjacent in $\R^d$. Take some $C>0$ sufficiently small, such that
$$
0<C<\min \left\{C_1, \frac{D_1}{4}, \frac{D_2}{4}, \frac{1}{10} \right\}
$$
Now for any cube $Q$ in $\R^d$, let $m_0 \in \Z$ such that
$$
\frac{C}{n^{m_0+1}} \le \ell(Q)<\frac{C}{n^{m_0}}.
$$
We consider several cases.

\medskip

\textit{Case I: $m_0>0$.} Let us show that $Q$ is contained in some cube in $\calG_{k, m_0}$ for some $k \in \{1, \dots, d+1\}$. We argue by contradiction. If $Q$ is not contained in any cubes in $\calG_{k, m_0}$ for all $k=1, \dots, d+1$. Then for each $k \in \{1, \dots, d+1\}$, there exists some $j_k \in \{1, \dots, d\}$, such that
$$
P_{j_k} (Q) \cap P_{j_k} (\calA_{k, m_0}) \neq \emptyset,
$$
where we recall $\calA_{k, m_0}$ is the collection of all the vertices of the cubes in $\calG_{k, m_0}$. By pigeonholing, there exists some $\ell_1, \ell_2 \in \{1, \dots, d+1\}$ with $\ell_1 \neq \ell_2$, but $j_*:=j_{\ell_1}=j_{\ell_2}$, such that
$$
P_{j_*} (Q) \cap P_{j_*} (\calA_{\ell_1, m_0}), \quad P_{j_*} (Q) \cap P_{j_*} (\calA_{\ell_2, m_0}) \neq \emptyset, 
$$
which implies that there exists some $K_1, K_2 \in \Z$, such that
$$
\left| \left(\del_{\ell_1}\right)_{j_*}+\frac{K_1}{n^{m_0}} -\left(\del_{\ell_2} \right)_{j_*}-\frac{K_2}{n^{m_0}} \right| \le \ell(Q)<\frac{C}{n^{m_0}}< \frac{C(\ell_1, \ell_2, j_*)}{n^{m_0}}, 
$$
which contradicts to \eqref{20200412eq101}. 

\medskip

\textit{Case II: $m_0 \le -N$.} Again, we wish to show that $Q$ is contained in some cubes in $\calG_{k, m_0}$ for some $k \in \{1, \dots, d+1\}$ and we prove it by contradiction. Following the argument in Case I above, we see that there exists some $k_1, k_2 \in \{1, \dots, d+1\}$ with $k_1 \neq k_2$ and $j^* \in \{1, \dots, d\}$, such that 
$$
P_{j^*} (Q) \cap P_{j^*} (\calA_{k_1, m_0}), \quad P_{j_*} (Q) \cap P_{j^*} (\calA_{k_2, m_0}) \neq \emptyset.
$$
This implies there exists some $K_3, K_4 \in \Z$, such that
$$
\left| \left[\del_{k_1}+\calL_{\bfVa_{k_1}}(-m_0) \right]_{j^*}+\frac{K_3}{n^{m_0}}- \left[\del_{k_2}+\calL_{\bfVa_{k_2}}(-m_0) \right]_{j^*}-\frac{K_4}{n^{m_0}} \right|<\frac{C}{n^{m_0}}. 
$$
Note that since we can always choose $N$ sufficiently large, we can indeed reduce the above estimate to
$$
\left| \left[\calL_{\bfVa_{k_1}}(-m_0) \right]_{j^*}+\frac{K_3}{n^{m_0}}- \left[\calL_{\bfVa_{k_2}}(-m_0) \right]_{j^*}-\frac{K_4}{n^{m_0}} \right|<\frac{2C}{n^{m_0}},
$$
which implies
\begin{equation} \label{20200412eq103}
\left| \frac{ \left[\calL_{\bfVa_{k_1}}(-m_0) \right]_{j^*}- \left[\calL_{\bfVa_{k_2}}(-m_0) \right]_{j^*}}{n^{-m_0}}+(K_3-K_4) \right|< 2C. 
\end{equation}
We claim that $K_3-K_4 \in \{-1, 0, 1\}$. Indeed, by the choice of $C$, the right hand side of the above estimate is bounded by $1/5$; on the other hand, by the definition of the location function, we have  
$$
\frac{ \left[\calL_{\bfVa_{k_1}}(-m_0) \right]_{j^*}- \left[\calL_{\bfVa_{k_2}}(-m_0) \right]_{j^*}}{n^{-m_0}} \in (-1, 1). 
$$
The desired claim then follows from these observations and the fact that $K_3-K_4$ is an integer. Therefore, the estimate \eqref{20200412eq103} implies either
$$
\left| \frac{ \left[\calL_{\bfVa_{k_1}}(-m_0) \right]_{j^*}- \left[\calL_{\bfVa_{k_2}}(-m_0) \right]_{j^*}}{n^{-m_0}} \right|< 2C<\frac{D_1}{2}
$$
or
$$
\left| \frac{ \left[\calL_{\bfVa_{k_1}}(-m_0) \right]_{j^*}- \left[\calL_{\bfVa_{k_2}}(-m_0) \right]_{j^*}}{n^{-m_0}} \right|>  1-2C>1-\frac{D_2}{2},
$$
which contradicts \eqref{20200412eq103}. 

\medskip

\textit{Case III: $-N<m_0 \le 0$.} Indeed, we can ``pass" the third case to the second case, by taking a cube $Q'$ containing $Q$ with the sidelength is $n^N$. Applying the second case to $Q'$, we find that there exists some $D \in \calG_k$ for some $k \in \{1, \dots, d+1\}$, such that $Q' \subset D$ and $\ell(D) \le C_4 \ell(Q')$, which clearly implies 
\begin{enumerate}
    \item [(1).] $Q \subset D$;
    \item [(2).] $\ell(D) \le C_4 n^N \ell(Q)$. 
\end{enumerate}

\medskip

The proof is complete. \hfill $\square$

\bigskip
\section{An illustrated example}

We now take some time to illustrate the effects of this theorem with a concrete example.  To begin with, let $\del_1 = \left(\frac{1}{3}, \frac{1}{3} \right), \del_2 = \left( \frac{2}{3}, \frac{2}{3} \right)$ and $\del_3 = \left(\frac{1}{5}, \frac{1}{5} \right)$, $d=2$ and $n=2$.  Also define the location functions via 
$$
\bfVa_1:= \begin{pmatrix}
1 & 0 & 1 & 0 & \dots\\
1 & 0 & 1 & 0 & \dots
\end{pmatrix},
$$
$$
\bfVa_2:=\begin{pmatrix}
0 & 1 & 0 & 1 & \dots\\
0 & 1 & 0 & 1 & \dots
\end{pmatrix}, 
$$
and 
$$
\bfVa_3:=\begin{pmatrix}
0 & 0 & 0 & 0 & \dots\\
0 & 0 & 0 & 0 & \dots
\end{pmatrix}.
$$
We will show that the grids $\calG_1 = \calG(\del_1, \calL_{\bfVa_1})$, $\calG_2 = \calG(\del_2, \calL_{\bfVa_2})$ and $\calG_3 = \calG(\del_3, \calL_{\bfVa_3})$ form adjacent dyadic systems in $\R^2$. Note that this example is optimal, in the sense that any two dyadic systems are not adjacent in $\R^2$. 

We start by verifying Condition (i) in Theorem \ref{mainresult01}, which is straightforward. Clearly, it suffices to show the numbers
$$
\frac{2}{3}-\frac{1}{3}=\frac{1}{3}, \quad \frac{1}{3}-\frac{1}{5}=\frac{2}{15}, \quad \textrm{and} \quad \frac{2}{3}-\frac{1}{5}=\frac{7}{16}
$$
are $2$-far (in the sense of Definition \ref{20200411defn01}). This is an easy exercise due to \cite[Proposition 2.4]{AHJOW} (see, also \cite[Lemma 3]{A}). 

While for the second condition, let us compute all the location functions. Indeed, the cases $j=1,2$ provide the key for the computations.  
$$
\del_1+\calL_{\bfVa_1}(j)= \begin{cases} 
\left(\frac{2^{j}}{3}, \frac{2^{j}}{3} \right),  & \hfill j \ \textrm{even} \\
\\
\left(\frac{2^{j+1}}{3}, \frac{2^{j+1}}{3} \right),  & \hfill j \ \textrm{odd}
\end{cases}, 
\quad 
\del_2+\calL_{\bfVa_2}(j)= \begin{cases} 
\left(\frac{2^{j+1}}{3}, \frac{2^{j+1}}{3} \right), & \hfill j  \ \textrm{even} \\
\\
\left(\frac{2^{j}}{3}, \frac{2^{j}}{3} \right), & \hfill j \ \textrm{odd}
\end{cases}, 
$$
and 
$$
\del_3+\calL_{\bfVa_3}(j) \equiv \left(\frac{1}{5}, \frac{1}{5} \right), \quad j \ge 1. 
$$
The second condition can be easily verified and we would like to leave the detail to the interested reader.

\part{A geometric approach}

In the second part of this paper, we provide an alternative way, based on the underlying geometry of $n$-adic systems, to generalize the one dimensional result Theorem \ref{dim1thm}. This approach is much more intuitive and unifies the proof of both the cases $d=1$ and $d>1$.

\section{Notations}

For $A, B \subset \R^d$, we write the \emph{distance} between them by
$$
\textrm{dist}(A, B):=\inf_{x \in A, y \in B} \|x-y\|_{\R^d}.
$$

\begin{defn}
Let $x \in \R^d$ and $A \subseteq \R^d$, the \emph{natural deviation} between $x$ and $A$ is defined to be 
$$
\textrm{dev}(A, x):=\max_{1 \le k \le d} \textrm{dist}_{\Vec{e}_k} (A, x),
$$
where $\textrm{dist}_{\Vec{e}_k}(A, x)$ denotes the distance between $x$ and $A$ along the direction $\Vec{e}_k$.
\end{defn}

\begin{rem}
Let us make some remarks for the above definition. 
\begin{enumerate}
    \item [(1).] The word ``natural" refers to the fact that we take the natural basis
    $$
    \{\Vec{e}_1, \dots, \Vec{e}_d\}
    $$
    in the definition. In general, one can replace $\{\Vec{e}_1, \dots, \Vec{e}_d\}$ by any other basis in $\R^d$, however, it is enough to take the natural basis in this paper;
    \item [(2).] The word ``deviation" refers to the fact that we take the maximal directional distance along all the directions $\Vec{e}_1, \dots \Vec{e}_d$;
    \item [(3).] In our application later, $A$ will be either a corner set (see, Definition \ref{cornerset}), a modulated corner set (see, Definition \ref{defncorner}), or a union of them. 
\end{enumerate}
\end{rem}

Let us include an easy example for these definitions (see, Figure \ref{Fig7}). In this example, we consider the case $d=2$, $x$ is the point $\left(\frac{1}{2}, \frac{3}{10} \right)$ and $A$ is the red part (which we will refer as a corner set later). Then it is clear that
$$
\text{dist}(x, A)=\frac{3}{10}-\frac{1}{10}=\frac{1}{5} \quad \textrm{and} \quad \text{dev}(x, A)=\frac{1}{2}-\frac{1}{6}=\frac{1}{3}.
$$

\begin{center}
\begin{figure}[ht]
\begin{tikzpicture}[scale=5]
\draw (0, -.3) [->] -- (0,.8);
\draw (-.3, 0) [->] -- (.8, 0);
\draw [opacity=2, red]  (1/6, 2/3)--(1/6, 1/10)--(2/3, 1/10);
\fill [opacity=1, red] (1/6, 1/10) circle [radius=.2pt];
\fill [opacity=1, red] (2/3, 1/10) circle [radius=.2pt];
\fill [opacity=1, red] (1/6, 2/3) circle [radius=.2pt];
\draw [opacity=2, dashed] (0, 2/3)--(1/6, 2/3); 
\fill (0, 2/3) node [left] {$\frac{2}{3}$};
\draw [opacity=2, dashed] (0, 1/10)--(1/6, 1/10); 
\fill (0, 1/6) node [left] {$\frac{1}{10}$};
\draw [opacity=2, dashed] (1/6, 1/10)--(1/6, 0);
\fill (1/6, 0) node [below] {$\frac{1}{6}$};
\draw [opacity=2, dashed] (2/3, 1/10)--(2/3, 0);
\fill (2/3, 0) node [below] {$\frac{2}{3}$};
\fill [opacity=2] (0.5, 0.3) circle [radius=.2pt];
\fill (.5, .3) node [right] {$\left(\frac{1}{2}, \frac{3}{10} \right)$};
\end{tikzpicture}
\caption{Distance and natural deviation}
\label{Fig7}
\end{figure}
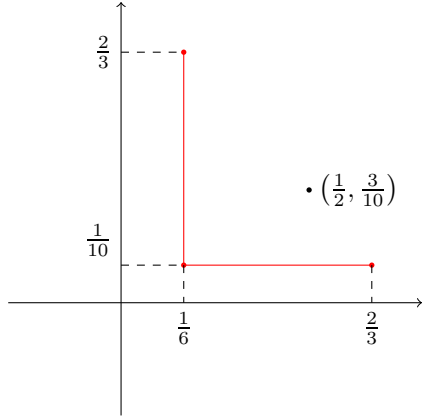
\end{center}

We frequently refer to the offspring and ancestor generations of a given dyadic cube; we often use the letter $m$ to refer to all generations while letter $j$ to specifically reserved for ancestors, this use comes from the roles of $m$ and $j$ in the \emph{shift} and \emph{location}, described below.

\medskip

\section{Fundamental structures of $d+1$ $n$-adic systems in $\R^d$}

In this section, we introduce several basic structures of a collection of $d+1$ $n$-adic systems, where we recall that $d+1$ is optimal in the sense that any $d$ $n$-adic systems in $\R^d$ are not adjacent, while there exists a collection of $d+1$ $n$-adic grids, which are adjacent. Using these basic structure, we are able to to generalize Theorem \ref{dim1thm} to higher dimensions in a natural way. 

This section consists three parts, which are the all generations case, the small scale case and the large scale case. Here, the word ``small scale" refers to all the offspring generations of the $0$-th generation; while the word ``large scale" refers to all the ancestor generations of the $0$-th generation, together with itself. Moreover, we also introduce the concept of $n$-far vector, which generalize the early definition of $n$-far number in one dimensional case (see, Definition \ref{20200411defn01}). Finally, there are also many concrete examples given for the purpose of understanding these structures better. 

Here is a list on all the structures that we are going to introduce in this section.

\begin{enumerate}

\item[$\bullet$]\textit{All generations case (Section 7.1):} Corner sets (see, Definition \ref{cornerset});

\medskip

\item[$\bullet$]\textit{Small scale case (Section 7.2):} Small scale lattice (see, Definition \ref{SSLat}), $n$-far vectors (see, Definition \ref{far number generalization});

\medskip

\item[$\bullet$]\textit{Large scale case (Section 7.3):} Large scale sampling (see, Definition \ref{largessampling}), Large scale lattice (see, Definition \ref{largessc}), Modulated corner sets (see, Definition \ref{defncorner}). 

\end{enumerate}

\subsection{Corner sets}

In the first part of this section, let us introduce an important structure for $\calG(\del, \calL_{\bfVa})$, which can be viewed as a ``generator" of $\calG(\del, \calL_{\bfVa})$.

\begin{defn} \label{cornerset}
Let $\calG(\del, \calL_{\bfVa})$ be a $n$-adic grid in $\R^d$. For $\ell  \in \Z$, we define the \emph{$\ell$-th corner operation}
$$
\calC_{\bfVa}^{\del}(\ell) :\calA(\del, \calL_{\bfVa})_{\ell} \to \R^d
$$
is given by 
$$
\left[\calC_{\bfVa}^{\del}(\ell)\right] (x):= \bigcup_{i=1}^d \prod_i \left(x, \left[(x)_1, (x)_1+n^{-\ell} \right) \times  \left[(x)_d, (x)_d+n^{-\ell} \right)  \right), 
$$
where for any $x \in \R^d$ and $A \subset \R^d$, $\prod_i(x, A)$ denotes the projection of $A$ to the hyperplane $\{y \in \R^d: (y)_i=(x)_i \}$. We call the collection $\left[\calC_{\bfVa}^{\del}(\ell)\right] (x)$ \emph{the $\ell$-th corner set associated to $x$ with parameters $(\del, \bfVa)$}, and we refer $x$ as the \emph{corner} of the corner set $\left[\calC_{\bfVa}^{\del} \right](x)$. 
\end{defn}

Below is an example of a corner set with its corner $x = \left(\frac{1}{3}, \frac{1}{3} \right)$ when $d=2, n=2$ and $\ell=1$ (see, Figure \ref{Fig4}).

\begin{center}
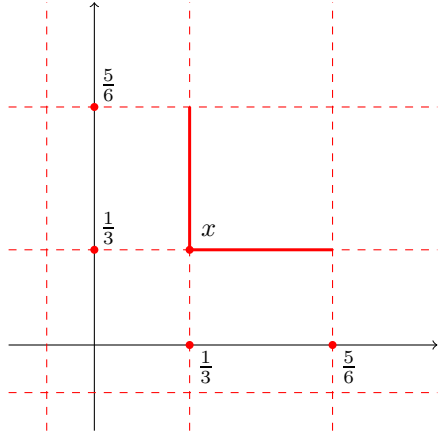
\begin{figure}[ht]
\begin{tikzpicture}[scale=3.8]
\draw (0, -.3) [->] -- (0,1.2);
\draw (-.3, 0) [->] -- (1.2,0);
\fill [opacity=1, red] (1/3, 1/3) circle [radius=.4pt];
\draw [opacity=1, red, dashed]  (1/3, -.3)--(1/3, 1.2);
\draw [opacity=1, red, dashed]  (-.3, 1/3)--(1.2, 1/3);
\draw [opacity=1, red, dashed]  (-.3, 5/6)--(1.2, 5/6);
\draw [opacity=1, red, dashed]  (5/6, -.3)--(5/6, 1.2);
\draw [opacity=1, red, dashed]  (-1/6, -.3)--(-1/6, 1.2);
\draw [opacity=1, red, dashed]  (-.3, -1/6)--(1.2, -1/6);
\draw [opacity=2, line width=0.4mm, red] (1/3, 1/3)--(5/6, 1/3);
\draw [opacity=2, line width=0.4mm, red] (1/3, 1/3)--(1/3, 5/6);
\fill (.4, .35) node [above] {$x$};
\fill (.39, .01) node [below] {$\frac{1}{3}$};
\fill (-.013, .41) node [right] {$\frac{1}{3}$};
\fill (.89, .01) node [below] {$\frac{5}{6}$};
\fill (-.013, .91) node [right] {$\frac{5}{6}$};
\fill [opacity=1, red] (1/3, 0) circle [radius=.4pt];
\fill [opacity=1, red] (0, 1/3) circle [radius=.4pt];
\fill [opacity=1, red] (5/6, 0) circle [radius=.4pt];
\fill [opacity=1, red] (0, 5/6) circle [radius=.4pt];
\end{tikzpicture}
\caption{Corner and Corner set}
\label{Fig4}
\end{figure}
\end{center}

\subsection{Small scale lattices and $n$-far vectors}

The goal of the second part of section is to generalize the concept of far numbers in one dimension to higher dimension. It turns out that the correct thing to do is looking at the so-called \emph{small scale lattices} with respect to $d$ many vectors in $\R^d$. We point out that such a construction is implicitly mentioned in Conde Alonso's proof of showing that $d$ grids are never adjacent \cite{Conde}.  

\begin{defn}
Let $\del \in \R^d$, we say $\del$ is of \emph{unit size} if $(\del)_i \in [0, 1)$ for each $i \in \{1, \dots, d\}$. 
\end{defn}

Let us first state an easy fact for the adjacent $n$-adic systems.
 
\begin{lem} \label{20200219lem01}
Let $\calG_1, \dots, \calG_{d+1}$ be adjacent, with each of them having a representation
$$
\calG_i=\calG(\del_i, \calL_{\bfVa_i}), \quad 1 \le i \le d+1.
$$
Then for each $k \in \{1, \dots, d\}$, 
$$
(\del_i)_k \not\equiv (\del_j)_k \quad  (\textrm{mod} \ 1), \quad  1 \le i \neq j \le d+1.
$$
\end{lem}

\begin{proof}
To start with, we may assume that all $\del_i$'s are of unit size. Otherwise, we consider $\del_i':=\del_i \quad (\textrm{mod} \ 1)$, where the modulus $1$ is taken over all the coordinate components. 

We prove it by contradiction. Without loss of generality, we may assume
$$
A:=(\del_1)_1=(\del_2)_1. 
$$
This implies that 
$$
\left\{ x \in \R^d: x_1=A \right\} \subset b\calG_{1, 0} \cap b\calG_{2, 0},
$$
where $b\calG_{1, 0}$ means the boundary of all cubes in the $0$-th generation of $\calG_1$. This further shows that for any $m \ge 0$,  the set
$$
\bigcap_{i=1}^{d+1} b\calG_{i, m} \quad \textrm{is not empty.}
$$
Indeed, it suffices to show that
\begin{equation} \label{20200219eq01}
\bigcap_{i=1}^{d+1} b\calG_{i, 0} \quad \textrm{is not empty.}
\end{equation}
since
$$
\bigcap_{i=1}^{d+1} b\calG_{i, 0} \subset \bigcap_{i=1}^{d+1} b\calG_{i, m}
$$
for all $m \ge 0$. While for a proof of \eqref{20200219eq01}, it suffices to note that the set
$$
\left\{x \in \R^d: x_1=A, x_2=(\del_3)_2, \dots, x_d=(\del_{d+1})_d \right\} \subset \bigcap_{i=1}^{d+1} b\calG_{i, 0}.
$$ 
To complete the proof, it suffices to follow Conde Alonso's argument. Namely, we pick a point $a \in \bigcap\limits_{i=1}^{d+1} b\calG_{i, 0}$ and we take a small cube centered at $a$. It is clear that the only way to cover this ball by using the cubes from $\bigcup\limits_{i=1}^{d+1} \calG_i$ is to use a cube whose sidelength is at least $n$. Shrinking the sidelength of the cube as small as we want, this gives a contradiction to the second property in Definition \ref{repdyadic}.
\end{proof}

Therefore, this easy lemma suggests the following definition. 

\begin{defn}
Given $\del_1, \dots, \del_\ell \in \R^d$, where $\ell \ge 1, \ell \in \Z$,  we say that they are \emph{separated} if 
$$
(\del_i)_k \neq (\del_j)_k \quad  (\textrm{mod} \ 1), 
$$
for all $k \in \{1, \dots, d\}$ and $1 \le i \neq j \le \ell$. 
\end{defn}

Now we are ready to introduce the so-called small scale lattices.

\begin{defn} \label{SSLat}
Let $m \ge 0$ and $\del_1, \dots, \del_d \in \R^d$ be separated. Then the \emph{small scale lattice associated to $\del_1, \dots, \del_d$ at level $m$} is defined by
$$
\frakL(\del_1, \dots, \del_d; m):=\bigcap_{i=1}^d b\left[\calG(\del_i)_m\right], 
$$
where $\calG(\del)_m$ is defined as in Definition \ref{repsgrids} and $b\left[\calG(\del)_m\right]$ refers to the union of all the boundary of the cubes in $\calG(\del)_m$.
\end{defn}

\begin{exa}
Let us give an example of such a structure (see, Figure \ref{Fig2}), in which, we consider the set $\calL\left( \left(\frac{1}{3}, \frac{1}{3}\right), \left(\frac{2}{3}, \frac{2}{3}\right) ; 1 \right)$ with $d=2$, $n=2$. More precisely, in Figure \ref{Fig2}, the red dashed lattice represents $b \left[\calG \left( \left(\frac{1}{3}, \frac{1}{3} \right) \right)_1 \right]$, the blue dashed lattice represents $b \left[\calG \left( \left(\frac{2}{3}, \frac{2}{3} \right) \right)_1 \right]$, and the 
collection of all green points stands for $\calL\left( \left(\frac{1}{3}, \frac{1}{3}\right), \left(\frac{2}{3}, \frac{2}{3}\right) ; 1 \right)$. 

\begin{center}
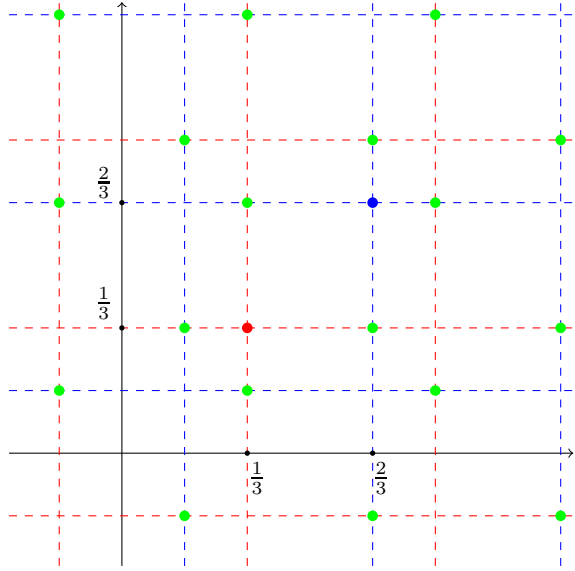
\begin{figure}[ht]
\begin{tikzpicture}[scale=5]
\draw (0, -.3) [->] -- (0,1.2);
\draw (-.3, 0) [->] -- (1.2,0);
\fill [opacity=1, red] (1/3, 1/3) circle [radius=.4pt];
\draw [opacity=1, red, dashed]  (1/3, -.3)--(1/3, 1.2);
\draw [opacity=1, red, dashed]  (-.3, 1/3)--(1.2, 1/3);
\draw [opacity=1, red, dashed]  (-.3, 5/6)--(1.2, 5/6);
\draw [opacity=1, red, dashed]  (5/6, -.3)--(5/6, 1.2);
\draw [opacity=1, red, dashed]  (-1/6, -.3)--(-1/6, 1.2);
\draw [opacity=1, red, dashed]  (-.3, -1/6)--(1.2, -1/6);
\fill (.36, 0)  node [below] {$\frac{1}{3}$};
\fill (0, .4)  node [left] {$\frac{1}{3}$};
\fill [opacity=1, blue] (2/3, 2/3) circle [radius=.4pt];
\draw [opacity=1, blue, dashed]  (2/3, -.3)--(2/3, 1.2);
\draw [opacity=1, blue, dashed]  (-.3, 2/3)--(1.2, 2/3);
\draw [opacity=1, blue, dashed]  (-.3, 7/6)--(1.2, 7/6);
\draw [opacity=1, blue, dashed]  (7/6, -.3)--(7/6, 1.2);
\draw [opacity=1, blue, dashed]  (1/6, -.3)--(1/6, 1.2);
\draw [opacity=1, blue, dashed]  (-.3, 1/6)--(1.2, 1/6);
\fill (.69, 0) node [below] {$\frac{2}{3}$};
\fill (0, .72) node [left] {$\frac{2}{3}$}; 
\fill [opacity=1, green] (1/3, 2/3) circle [radius=.4pt];
\fill [opacity=1, green] (2/3, 1/3) circle [radius=.4pt];
\fill [opacity=1, green] (7/6, 1/3) circle [radius=.4pt];
\fill [opacity=1, green] (5/6, 2/3) circle [radius=.4pt];
\fill [opacity=1, green] (5/6, 7/6) circle [radius=.4pt];
\fill [opacity=1, green] (1/3, 7/6) circle [radius=.4pt];
\fill [opacity=1, green] (1/3, 1/6) circle [radius=.4pt];
\fill [opacity=1, green] (5/6, 1/6) circle [radius=.4pt];
\fill [opacity=1, green] (2/3, -1/6) circle [radius=.4pt];
\fill [opacity=1, green] (7/6, -1/6) circle [radius=.4pt];
\fill [opacity=1, green] (7/6, 5/6) circle [radius=.4pt];
\fill [opacity=1, green] (2/3, 5/6) circle [radius=.4pt];
\fill [opacity=1, green] (1/6, 5/6) circle [radius=.4pt];
\fill [opacity=1, green] (1/6, 1/3) circle [radius=.4pt];
\fill [opacity=1, green] (1/6, -1/6) circle [radius=.4pt];
\fill [opacity=1, green] (-1/6, 1/6) circle [radius=.4pt];
\fill [opacity=1, green] (-1/6, 2/3) circle [radius=.4pt];
\fill [opacity=1, green] (-1/6, 7/6) circle [radius=.4pt];
\fill [opacity=1] (1/3, 0) circle [radius=.2pt];
\fill [opacity=1] (2/3, 0) circle [radius=.2pt];
\fill [opacity=1] (0, 2/3) circle [radius=.2pt];
\fill [opacity=1] (0, 1/3) circle [radius=.2pt];
\end{tikzpicture}
\caption{$\calL\left( \left(\frac{1}{3}, \frac{1}{3}\right), \left(\frac{2}{3}, \frac{2}{3}\right) ; 1 \right)$}
\label{Fig2}
\end{figure}
\end{center}
\end{exa}

\begin{rem}
\begin{enumerate}
    \item [(1).] We call it ``small scale", in the sense that set $\frakL(\del_1, \dots, \del_d; m)$ is constructed with respect to the boundary of cubes of small sidelength;
    \item [(2).] Since $\del_1, \dots, \del_d$ are separated, $\frakL(\del_1, \dots, \del_d; m)$ is countable (it contains no lines);
    \item [(3).] $\frakL(\del_1, \dots, \del_d; m) \subset \frakL(\del_1, \dots, \del_d; m')$ for all $m' \ge m$.
\end{enumerate}
\end{rem}

There is an equivalent way to define $\calL(\del_1, \dots, \del_d; m)$.

\begin{lem} \label{equivsmallscalelattice}
Let $\del_1, \dots, \del_d$ be separated. For each $\sigma \in S_d$, where $S_d$ refers to the finite symmetric group over $\{1, \dots, d\}$, denote
$$
\del_\sigma:=\left( \left(\del_{\sigma(1)} \right)_1, \dots,  \left(\del_{\sigma(d)} \right)_d \right) \in \R^d.
$$
Then for each $m \ge 0$, there holds
$$
\calL(\del_1, \dots, \del_d; m)=\bigcup_{\sigma \in S_d} \calA(\del_\sigma)_m.
$$
\end{lem}

\begin{proof}
The desired result follows from the observation that $\del_\sigma \in \calL(\del_1, \dots, \del_d; m)$ and hence $\calA(\del_\sigma)_m \subset \calL(\del_1, \dots, \del_d; m)$. One may consult Figure \ref{Fig2} for an example of such a fact.
\end{proof}

Let us denote
$$
\frakL(\del_1, \dots, \del_d):=\bigcup_{m \ge 0} \frakL(\del_1, \dots, \del_d; m). 
$$
be the union of all small scale lattices.

We are ready to generalize the concept of far number in dimension one case.

\begin{defn}
\label{far number generalization}
Given $\del_1, \dots, \del_d \in \R^d$ being separated, we say $\del \in \R^d$ is a \emph{$n$-far vector} with respect to $\frakL(\del_1, \dots, \del_d)$, if there exists some $C>0$, such that for every $m \ge 0$, there holds
\begin{equation} \label{20200319eq02}
\textrm{dist} \left(b\left[\calG(\del)_m\right],  \frakL(\del_1, \dots, \del_d; m) \right) \ge \frac{C}{n^m}.
\end{equation}
\end{defn}

\begin{rem} \label{20200320rem01}

While the above definition looks quite different from Definition 1.3 of \cite{AHJOW}, we can actually motivate our definition from \cite{AHJOW}. In the one-dimensional case, the set $\{\del_1, \dots, \del_d\}$ reduces to the set $\{0\}$. Therefore, $n$-far in the one dimensional case reduced to checking that
\begin{equation} \label{20200319eq01}
\left|\del-\frac{k}{n^m} \right| \ge \frac{C}{n^m},
\end{equation}
for $k \in \Z$ and $m \ge 0$ (see Remark 1.4 (2) in \cite{AHJOW}). Due to such a simple structure, Theorem 2.8 in \cite{AHJOW} asserted that $\del \in \R$ is $n$-far if and only if the maximal length of tie, that is, the maximal length of consecutive $0$'s or consecutive $n-1$'s in the base $n$-representation of $\del$, is finite. 

The key point to generalize the concept of $n$-far number in higher dimension is to realize that $\del$ and $\frac{k}{n^m}$ are indeed playing different roles in condition \eqref{20200319eq01}. More precisely, rewrite \eqref{20200319eq01} as
$$
\left| \left(\del-\frac{k_1}{n^m} \right)-\frac{k_2}{n^m} \right| \ge \frac{C}{n^m}
$$
for $k_1, k_2 \in \Z$ and $m \ge 0$. Here are the key observations: let $m \ge 0$ be fixed, then 
\begin{enumerate}
    \item [(1).] when $k_1$ varies, the set
    $$
    \left\{ \del-\frac{k_1}{n^m}: k_1 \in \Z \right\}
    $$
    represents the all the boundary points $b\left[\calG(\del)_m \right]$;
    
    \medskip
    
    \item [(2).] when $k_2$ varies, the set
    $$
    \left\{\frac{k_2}{n^m}: k_2 \in \Z \right\}
    $$
    represents the small scale lattice $\calL(0; m)$. This is clear from Lemma \ref{equivsmallscalelattice}. 
\end{enumerate}

Hence, an equivalent way to state \eqref{20200319eq01} is the following: for any $m \ge 0$, 
$$
\textrm{dist} \left( b\left[ \calG(\del)_m \right], \calL(0; m) \right) \ge \frac{C}{n^m},
$$
which is precisely \eqref{20200319eq02}.  The key here is that the boundary points and the small scale lattice align in dimension 1.
\end{rem}

\begin{exa}
 We now illustrate via an example why the idea of generalizing the concept of far numbers to pairwise distances is not sufficient to show that $d+1$ grids are adjacent.  Take $\del_1 = (0,0), \del_2 = \left(0,\frac{1}{3} \right)$ and $\del_3 = \left(\frac{1}{3},\frac{1}{3} \right)$.  Consider the natural pairwise generalization of far where $\del$ and $\del'$ are far if $(\del)_i$ is far from $(\del')_i$ for some $i$.  Using this definition, we see that each of these points is far from each other, but $\del_1, \del_2, \del_3$ do no form an adjacent $n$-adic system.  In fact, if one looks at $\mathcal{L}(\del_1, \del_2;0)$, one does not get a lattice, but a set of vertical lines (see Figure \ref{Fig8}).

 Therefore this pairwise comparison between points is not enough to determine an adjacent $n$-adic system and the lattice structure is needed.
\end{exa}

\begin{center}
\begin{figure}[ht]
\begin{tikzpicture}[scale=5.5]
\draw (0, -.3) [->] -- (0, .65);
\draw (-.3, 0) [->] -- (.65,0);
\draw [opacity=1, red, dashed]  (1/3, -.3)--(1/3, .6);
\draw [opacity=1, red, dashed]  (-.3, 1/3)--(.6, 1/3);
\draw [opacity=1, red, dashed]  (-1/6, -.3)--(-1/6, .6);
\draw [opacity=1, red, dashed]  (-.3, -1/6)--(.6, -1/6);
\fill (.286, 1/3) node [below] {$\del_3$};
\draw [opacity=1, dashed]  (.5, -.3)--(1/2, .6);
\draw [opacity=1, dashed]  (-.3, 1/2)--(.6, 1/2);
\fill (-.04, 0) node [below] {$\del_1$};
\draw [opacity=1, blue, dashed]  (1/2, -.3)--(1/2, .6);
\draw [opacity=1, blue, dashed]  (-.3, 1/3)--(.6, 1/3);
\draw [opacity=1, blue, dashed]  (0, -.3)--(0, .6);
\draw [opacity=1, blue, dashed]  (-.3, -1/6)--(.6, -1/6);
\fill (-.04, 1/3) node [below] {$\del_2$};
\draw [opacity=2, line width=0.4mm, green] (0, -.3)--(0, .6);
\draw [opacity=2, line width=0.4mm, green] (.5, -.3)--(.5, .6);
\fill [opacity=1, red] (1/3, 1/3) circle [radius=.4pt];
\fill [opacity=1] (0, 0) circle [radius=.4pt];
\fill [opacity=1, blue] (0, 1/3) circle [radius=.4pt];
\end{tikzpicture}
\caption{A non-example of small scale lattice: $b[\calG(\del_1)]_0$ (black part),  $b[\calG(\del_2)]_0$ (blue part), $b[\calG(\del_3)]_0$ (red part) and the non-example $b\left[\calG(\del_1)\right]_0 \cap b\left[\calG(\del_2)\right]_0$ (green part). }
\label{Fig8}
\end{figure}
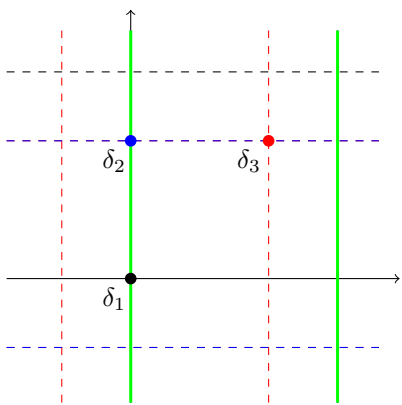
\end{center}


\subsection{Large scale sampling, large scale lattice and modulated corner sets}

The purpose of the last part of section is to introduce all the basic structures with respect to $d$ separated vectors and $d$ location functions, which is used to study the large scale case. Informally, these structures are the counterparts of the small scale lattice (see, Definition \ref{SSLat}) and the usual corner sets (see, Definition \ref{cornerset}), however, one new phenomenon for the large scale is that one needs to use the cube $[0, n^j)^d$ to quantify the behavior of the location function (see, Theorem \ref{mainresult}, (ii)), and this suggests that we need a localized version of those structures in the large scale case.

The setting is as follows. Let $\del_1, \dots, \del_d \in \R^d$ be separated, $\bfVa_1, \dots, \bfVa_d$ be the $d$ many infinite matrices defined in \eqref{20200226eq02}, and $\calL_{\bfVa_1}, \dots, \calL_{\bfVa_d}$ be the corresponding location functions.

\begin{defn} \label{largessampling}
For each $j>0$, the \emph{large scale sampling associated to the tuples
$$
\left(\del_1, \bfVa_1\right), \dots, \left(\del_d, \bfVa_d\right)
$$
at level $j$} is defined by
$$
\calS_{\bfVa_1, \dots, \bfVa_d}^{\del_1, \dots, \del_d}(j):=\left( \bigcap_{k=1}^d b\left[ \calG(\del_k, \calL_{\bfVa_k})_{-j}\right]\right) \cap [0, n^j)^d.
$$
\end{defn}

\begin{exa}
Let $d=2, n=2$, $\del_1 = \left(\frac{1}{3}, \frac{1}{3} \right)$ and $\del_2 = \left(\frac{2}{3}, \frac{2}{3} \right)$.  We illustrate below the large scale sampling (see, Figure \ref{Fig9}) for $j=1$ for $\left(\del_1, \bfVa_1 \right)$ and $\left(\del_2, \bfVa_2 \right)$ (in the cube $[0,2)^2)$) in Figure \ref{Fig9}, where
$$
\bfVa_1:= \begin{pmatrix}
1 & 0 & 1 & 0 & \dots\\
1 & 0 & 1 & 0 & \dots
\end{pmatrix}
$$
and 
$$
\bfVa_2:=\begin{pmatrix}
0 & 1 & 0 & 1 & \dots\\
0 & 1 & 0 & 1 & \dots
\end{pmatrix}
$$
\end{exa}

\begin{center}
\begin{figure}[ht]
\begin{tikzpicture}[scale=2]
\draw (0, -.3) [->] -- (0, 2.2);
\draw (-.3, 0) [->] -- (2.2,0);
\fill [opacity=1, red] (4/3, 4/3) circle [radius=1pt];
\fill [opacity=1, blue] (2/3, 2/3) circle [radius=1pt];
\draw [opacity=1, dashed]  (0, 2)--(2, 2)--(2, 0);
\fill (2, 0) node [below] {$2$};
\fill (0, 2) node [left]  {$2$};
\draw [opacity=1, dashed, red] (-.3, 4/3)--(2.2, 4/3);
\draw [opacity=1, dashed,  red] (4/3, -.3)--(4/3, 2.2); 
\draw [opacity=1, dashed, blue] (2/3, -.3)--(2/3, 2.2);
\draw [opacity=1, dashed, blue] (-.3, 2/3)--(2.2, 2/3); 
\fill (1.44, 4/3) node [above] {$A$};
\fill (.55, 2/3) node [below] {$B$};
\fill [opacity=1, green] (2/3, 4/3) circle [radius=1pt];
\fill [opacity=1, green] (4/3, 2/3) circle [radius=1pt];
\fill (1.43, 0) node [below] {$\frac{4}{3}$};
\fill (.75, 0) node [below] {$\frac{2}{3}$};
\fill (0, .5) node [left] {$\frac{2}{3}$};
\fill (0, 1.17) node [left] {$\frac{4}{3}$};
\end{tikzpicture}
\caption{An example of large scale sampling: $A=\del_1+\calL_{
\bfVa_1}(1)=\left(\frac{4}{3}, \frac{4}{3} \right)$, $B=\del_2+\calL_{\bfVa_2}(1)=\left(\frac{2}{3}, \frac{2}{3} \right)$, $b\left[\calG(\del_1, \bfVa_1) \right]_{-1}$ (red part), $b\left[\calG(\del_2, \bfVa_2) \right]_{-1}$ (blue part), and the large scale sampling $\calS_{\bfVa_1, \bfVa_2}^{\del_1, \del_2}(1)$ (green part). 
}
\label{Fig9}
\end{figure}
\end{center}

Let us extend the definition of small scale lattice to the large scales, which can be viewed as a ``global" version of the large scale sampling.
 
\begin{defn} \label{largessc}
For each $m<0$, the \emph{large scale lattice associated to the tuples
$$
\left(\del_1, \bfVa_1\right), \dots, \left(\del_d, \bfVa_d\right)
$$
at level $m$} is defined by
$$
\calL_{\bfVa_1, \dots, \bfVa_d}^{\del_1, \dots, \del_d}(m):=\bigcap_{k=1}^d b\left[ \calG(\del_k, \calL_{\bfVa_k})_m\right].
$$
\end{defn}

\begin{rem}
Again, we emphasize the use of $j=-m$ where $m<0$ and $j$ refers to the $j$-th ancestor (that is, the $-j$-th generation in Definition \ref{repsgrids}). We only use this identification for the large cubes, the reason we do so is to be consistent with the terminology in \cite{AHJOW}.
\end{rem}

\begin{rem}
\begin{enumerate}
    \item [1.] Note that since $\del_1, \dots, \del_d$ are assumed to be separated, $\calS_{\bfVa_1, \dots, \bfVa_d}^{\del_1, \dots, \del_d}(j)$ is a finite set, while $\calL_{\bfVa_1, \dots, \bfVa_d}^{\del_1, \dots, \del_d}(m)$ is countable. 
    \item [2.] The reason for us to start with a negative $m$ in the definition of large scale lattice is to keep the consistency with the definition of small scale lattice.
    \item [2.] The large scale sampling can be viewed as a ``generator" of the large scale lattice, more precisely, we have, for each $m>0$, there holds
    $$
    \calL_{\bfVa_1, \dots, \bfVa_d}^{\del_1, \dots, \del_d}(m)=\bigcup_{p \in \Z^d} \left(\calS_{\bfVa_1, \dots, \bfVa_d}^{\del_1, \dots, \del_d}(-m)+n^j p\right).
    $$
\end{enumerate}
\end{rem}

We need one more definition to state our main result in large scale.

\begin{defn} \label{defncorner}
Let $j \ge 0$, $\del \in \R^d$, $\bfVa$ be an infinite matrix defined in \eqref{20200226eq02} and $\calL_{\bfVa}$ be the associated location function. Then the \emph{$j$-th modulated corner set associated to $(\del, \bfVa)$} is defined to be 
$$
(m\calC)_{\bfVa}^{\del}(j):=\left( b \left[\calG(\del, \calL_\bfVa) \right]_{-j}\right) \cap [0, n^j)^d.
$$ 
\end{defn}

\begin{exa}
We give an easy example of such a structure (see, Figure \ref{Fig3}). Here we take our favorite example, that is, we consider the case when $d=2$, $n=2, j=1$,  $\del=\left(\frac{1}{3}, \frac{1}{3} \right)$, $\bfVa:= \begin{pmatrix}
1 & 0 & 1 & 0 & \dots\\
1 & 0 & 1 & 0 & \dots
\end{pmatrix}$.

\begin{center}
\begin{figure}[ht]
\begin{tikzpicture}[scale=2]
\draw (0, -.1) [->] -- (0, 2.2);
\draw (-.1, 0) [->] -- (2.2,0);
\draw [dashed] (2, 0)--(2, 2)--(0, 2); 
\fill [opacity=1, red] (4/3, 4/3) circle [radius=1pt];
\draw [line width=0.25mm, red ] (4/3, 0)--(4/3, 2);
\draw [line width=0.25mm, red ] (0, 4/3)--(2, 4/3); 
\fill (1.6, 1.4) node [above] {$\left(\frac{4}{3}, \frac{4}{3} \right)$};
\fill (2, 0) node [below] {$2$};
\fill (0, 2) node [left]  {$2$}; 
\end{tikzpicture}
\caption{$(m\calC)_{\bfVa}^{\left(\frac{1}{3}, \frac{1}{3} \right)}(1)$.}
\label{Fig3}
\end{figure}
\end{center}
\end{exa}

\begin{rem} \label{20200325rem01}
\begin{enumerate}
    \item [1.]  An trivial observation would be for any $y' \in \calS_{\bfVa_1, \dots, \bfVa_d}^{\del_1, \dots, \del_d}(j)$, there holds
$$
0 \le \frac{\textrm{dev} \left( (m \calC)_\bfVa^{\del}(j), y' \right)}{n^j}<1.  
$$
\item [2.] The corner sets and modulated corner sets are closely related. More precisely, we have for any $d \ge 1$ and $j \ge 0$, there holds
$$
\left[\calC_{\bfVa}^{\del}(-j) \right] \left(\del+\calL_{\bfVa}(j) \quad (\textrm{mod} \ n^j) \right) \equiv (m\calC)_{\bfVa}^{\del}(j) \quad (\textrm{mod} \ n^j)
$$

\end{enumerate}

\end{rem}

In our later application, the modulated corner sets will only be applied to study the ancestors of the $0$-th generation, while the corner sets defined early will take care of all the generations.  More precisely, the modulated corners are important to take care of points near the boundary of the cube $[0, n^j)^d$ that might be close to a corner lying just outside $[0,n^j)^d$ but far from the modulated corner. This remark will be made clearer in the next section and indeed underlies our use of modulated corners.

\section{An alternative approach}

Let us go back to the main question that we are interested in, namely, \emph{what the necessary and sufficient conditions are, so that a given collection of $d+1$ $n$-adic grids are adjacent?}

The goal of this section is to provide an alternative way to answer the above question via the fundamental structures, and to comment about the uniformity of such a representation. To begin, let us write these $d+1$ $n$-adic grids by their representations, namely $\calG(\del_1, \calL_{\bfVa_1}), \dots, \calG(\del_{d+1}, \calL_{\bfVa_{d+1}})$. Moreover, using Lemma \ref{20200219lem01}, we may assume $\del_1, \dots, \del_{d+1}$ are separated. 

\medskip
\subsection{Motivation and the structure theorem of adjacency}

As mentioned in the introduction, the intuition behind the next theorem originates back to the one dimensional result (see, Theorem \ref{dim1thm}). The main idea to generalize the first condition is already contained in Remark \ref{20200320rem01}. More precisely, we can rewrite condition (1) as: there exists some absolute constant $C>0$, such that for any $k_1, k_2 \in \Z$ and $m \ge 0$, it holds that
$$
\left| \left( \del_1-\frac{k_1}{n^m} \right)-\left( \del_2-\frac{k_2}{n^m} \right) \right| \ge \frac{C}{n^m}.
$$
For simplicity, we may assume both $\del_1, \del_2 \in [0, 1)$, that is, both of them are of unit size (although we do not require such a restriction in our main result). As in Remark \ref{20200320rem01}, there are two different ways to interpret the sets
$$
\left\{\del_i-\frac{k}{n^m}: k_i \in \Z \right\}, \quad i=1, 2. 
$$
The first way is to interpret each set as the set of all the boundary points $b\left[ \calG(\del_i)_m\right]$; while the second way is to treat it as the small scale lattice $\frakL(\del_i; m)$. Therefore, we can restate the first condition as 
\begin{center}
    \emph{$\del_1$ is $n$-far with respect to $\frakL(\del_2)$ and $\del_2$ is $n$-far with respect to $\frakL(\del_1)$}.
\end{center}

\medskip

While for condition (2), note that it is equivalent to the following: 
\begin{equation} \label{20200326eq11}
0<\liminf_{j \to \infty} \left| \frac{\del_1+\calL_{\bfVa_1}(j)-\del_2-\calL_{\bfVa_2}(j)}{n^j} \right| \le \limsup_{j \to \infty} \left| \frac{\del_1+\calL_{\bfVa_1}(j)-\del_2-\calL_{\bfVa_2}(j)}{n^j} \right|<1. 
\end{equation}
Similarly, there are two ways to interpret the sets (which contain only one element in the one dimensional case):
\begin{equation} \label{20200326eq10}
\left\{ \del_i+\calL_{\bfVa_i}(j) \right\}, \quad i=1, 2,
\end{equation}
for any $j \ge 1$. The first way is to regard each as a $j$-th modulated corner set, namely, $(m\calC)_{\bfVa_i}^{\del_i}(j)$; while the second way is to consider it as a large scale sampling at level $j$, that is, $\calS_{\bfVa_i}^{\del_i}(j)$.  If one sets the interpretation of the first set as the modulated corner and of the second set as the large scale sampling, then additionally there are also two different ways to understand the term
\begin{equation} \label{20200321eq01}
\left| \del_1+\calL_{\bfVa_1}(j)-\del_2-\calL_{\bfVa_2}(j) \right|.
\end{equation}
The first way is to treat it as
$$
\textrm{dist} \left( (m\calC)_{\bfVa_i}^{\del_i}(j), \calS_{\bfVa_{i'}}^{\del_{i'}}(j) \right),
$$
where $\{i, i'\}=\{1, 2\}$; while the second way is to regard it as the ``maximum distance" between these two sets, that is
$$
\max_{y' \in \calS_{\bfVa_{i'}}^{\del_{i'}}(j)} \textrm{dev} \left( (m\calC)_{\bfVa_i}^{\del_i}(j), y' \right). 
$$
Note that all these complicated structures (that is, the modulated corner set and large scale sampling) collapse into a single point in one dimensional case, and all those quantities coincide with each other and take the value \eqref{20200321eq01}, however, in higher dimension, we need to treat them differently. 

Here is the main result for the second part, which can be thought as some \emph{structure theorem of adjacency} with respect to a collection of $d+1$ $n$-adic grids. 

\begin{thm} \label{mainresult}
The collection of $n$-adic grids $\calG(\del_1, \calL_{\bfVa_1}), \dots, \calG(\del_{d+1}, \calL_{\bfVa_{d+1}})$ are adjacent if and only if 
\begin{enumerate}
    \item [(i).] For each $i \in \{1, \dots, d+1\}$, $\del_i $ is $n$-far with respect to $\frakL\left(\del_1, \dots, \widehat{\del_i}, \dots, \del_{d+1}\right)$. Here, $\widehat{\del_i}$ means that in the sequence $\{\del_1, \dots, \del_{d+1}\}$, we remove the term $\del_i$; 
    \item [(ii).] For each $i \in \{1, \dots, d+1\}$, there holds
\begin{eqnarray} \label{20200326eq12}
&& 0<\liminf_{j \to \infty} \frac{\textrm{dist} \left( (m\calC)_{\bfVa_i}^{\del_i} (j), \calS_{\bfVa_1,  \dots, \widehat{\bfVa_i}, \dots, \bfVa_{d+1}}^{\del_1, \dots, \widehat{\del_i}, \dots, \del_{d+1}}(j) \right)}{n^j} \nonumber \\ 
&&\quad \quad \quad \quad \quad  \quad \quad \quad \quad  \le \limsup_{j \to \infty} \frac{\max_{y' \in \calS_{\bfVa_1,  \dots, \widehat{\bfVa_i}, \dots, \bfVa_{d+1}}^{\del_1, \dots, \widehat{\del_i}, \dots, \del_{d+1}}(j)} \textrm{dev}\left( (m\calC)_{\bfVa_i}^{\del_i} (j), y' \right)}{n^j} 
<1. 
\end{eqnarray}
\end{enumerate}
Note that in the above statement, for each $j>0$ and $i \in \{1. \dots, d+1\}$, 
$$ 
\calS_{\bfVa_1,  \dots, \widehat{\bfVa_i}, \dots, \bfVa_{d+1}}^{\del_1, \dots, \widehat{\del_i}, \dots, \del_{d+1}}(j)
$$
is a finite set, and hence the maximum value in the limit superior can be attained. 
\end{thm}

\begin{rem} \label{20200410rem01}
Note that the above result is independent of the choice of the representations of $n$-grid and hence contains the uniformness result Theorem \ref{uniform} as a particular case. The reason is that the structures that we are working with, that is, the collection of all the boundary points at certain generations, the small scale lattices, the modulated corner sets and the large scale sampling are independent of the choice of the representations; while for the term 
$$
\del+\calL_{\bfVa}(j)
$$
(see, \eqref{20200326eq10} and \eqref{20200326eq11}) that we used in our one dimensional result indeed depends on the particular choices of the representation, more precisely, it is possible for us to find two equivalent presentations $\calG(\del, \calL_{\bfVa})$ and $\calG(\del', \calL_{\bfVa'})$ such that
$$
\del+\calL_{\bfVa}(j)<n^j \le \del'+\calL_{\bfVa'}(j)
$$
for $j$ sufficiently large. In such a situation, Theorem \ref{mainresult} suggests us to work with the point $\del'+\calL_{\bfVa'}(j)-n^j$ instead of $\del'+\calL_{\bfVa'}(j)$, and this will guarantee that there is only one limit infimum and one limit superior in \eqref{20200326eq12}, instead many of them. 
\end{rem}

\begin{proof} [Proof of Theorem \ref{mainresult}]
It turns out the conditions (i) and (ii) above are indeed equivalent to the conditions (1) and (2) in Theorem \ref{mainresult01}, which indeed suggests Theorem \ref{mainresult} is a much more natural statement to consider.  We would like to leave the detail to the interested reader. 
\end{proof}

\medskip
\subsection{Example revisited}

Let us use Theorem \ref{mainresult} to verify the adjacency of the dyadic grids considered in Section 5. We start by verifying Condition (i) in Theorem \ref{mainresult}.  This means we have to show
\begin{enumerate}
    \item [(a).] $\del_1$ is $2$-far with respect to $\frakL\left(\del_2, \del_3\right)$;
    \item [(b).] $\del_2$ is $2$-far with respect to $\frakL\left(\del_1, \del_3\right)$;
    \item [(c).] $\del_3$ is $2$-far with respect to $\frakL\left(\del_1, \del_2\right)$.
\end{enumerate}
We will show (c).  First, it is easy to see that $\{\del_1, \del_2, \del_3\}$ are separated. Note that calculating the quantity \eqref{20200319eq02} for $m=1$ reduces to calculating in a ``local region" (see, the cyan part in Figure \ref{Fig10}), as both $b[\calG(\del_3)]_1$ and $\frakL(\del_1, \del_2; 1)$ behave periodically. Therefore, we can see that

$$
\textrm{dist} \left(b\left[\calG(\del_3)\right]_1,  \frakL(\del_1,\del_2; 1) \right) = \min\left\{\frac{1}{3}-\frac{1}{5}, \ \frac{1}{5}-\frac{1}{6} \right\} = \frac{1/15}{2}.
$$

\begin{center}
\begin{figure}[ht]
\begin{tikzpicture}[scale=6]
\draw (0, -.5) [->] -- (0, 1.5);
\draw (-.5, 0) [->] -- (1.5,0);
\fill [opacity=1, red] (1/5, 1/5) circle [radius=.3pt];
\draw [line width=0.25mm,  red ] (-.5, .2)--(1.5, .2);
\draw [line width=0.25mm,  red ] (-.5, .7)--(1.5, .7);
\draw [line width=0.25mm,  red ] (-.5, 1.2)--(1.5, 1.2);
\draw [line width=0.25mm,  red ] (-.5, -.3)--(1.5, -.3);
\draw [line width=0.25mm, red ] (.2, -.5)--(.2, 1.5);
\draw [line width=0.25mm, red ] (.7, -.5)--(.7, 1.5);
\draw [line width=0.25mm, red ] (1.2, -.5)--(1.2, 1.5);
\draw [line width=0.25mm, red ] (-.3, -.5)--(-.3, 1.5);
\fill (.24, .21) node [above] {$\del_3$};
\fill [opacity=1, blue] (1/3, 1/3) circle [radius=.3pt];
\draw [line width=0.25mm, dashed, blue ] (-.5, 1/3)--(1.5, 1/3);
\draw [line width=0.25mm, dashed, blue ] (-.5, -1/6)--(1.5, -1/6);
\draw [line width=0.25mm, dashed, blue] (-.5, 5/6)--(1.5, 5/6);
\draw [line width=0.25mm, dashed, blue ] (-.5, 4/3)--(1.5, 4/3);
\draw [line width=0.25mm, dashed, blue ] (1/3, -.5)--(1/3, 1.5);
\draw [line width=0.25mm, dashed, blue] (5/6, -.5)--(5/6, 1.5);
\draw [line width=0.25mm, dashed, blue] (4/3, -.5)--(4/3, 1.5);
\draw [line width=0.25mm, dashed, blue] (-1/6, -.5)--(-1/6, 1.5);
\fill (.3, .33) node [below] {$\del_1$};
\fill [opacity=1, darkgray] (2/3, 2/3) circle [radius=.3pt];
\draw [line width=0.25mm, darkgray, dashed] (-.5, 2/3)--(1.5, 2/3);
\draw [line width=0.25mm, darkgray, dashed] (-.5, 1/6)--(1.5, 1/6);
\draw [line width=0.25mm, darkgray, dashed] (-.5, -1/3)--(1.5, -1/3);
\draw [line width=0.25mm, darkgray, dashed] (-.5, 7/6)--(1.5, 7/6);
\draw [line width=0.25mm, darkgray, dashed] (2/3, -.5)--(2/3, 1.5);
\draw [line width=0.25mm, darkgray, dashed] (1/6, -.5)--(1/6, 1.5);
\draw [line width=0.25mm, darkgray, dashed] (-1/3, -.5)--(-1/3, 1.5);
\draw [line width=0.25mm, darkgray, dashed] (7/6, -.5)--(7/6, 1.5);
\fill (.62, .66) node [below] {$\del_2$};
\fill [opacity=1, green] (1/3, 2/3) circle [radius=.4pt];
\fill [opacity=1, green] (2/3, 1/3) circle [radius=.4pt];
\fill [opacity=1, green] (1/6, 5/6) circle [radius=.4pt];
\fill [opacity=1, green] (5/6, 1/6) circle [radius=.4pt];
\fill [opacity=1, green] (2/3, 5/6) circle [radius=.4pt];
\fill [opacity=1, green] (5/6, 2/3) circle [radius=.4pt];
\fill [opacity=1, green] (1/6, 1/3) circle [radius=.4pt];
\fill [opacity=1, green] (-1/6, 1/6) circle [radius=.4pt];
\fill [opacity=1, green] (1/6, -1/6) circle [radius=.4pt];
\fill [opacity=1, green] (-1/6, 2/3) circle [radius=.4pt];
\fill [opacity=1, green] (2/3, -1/6) circle [radius=.4pt];
\fill [opacity=1, green] (-1/3, -1/6) circle [radius=.4pt];
\fill [opacity=1, green] (-1/6, -1/3) circle [radius=.4pt];
\fill [opacity=1, green] (-1/3, 1/3) circle [radius=.4pt];
\fill [opacity=1, green] (1/3, -1/3) circle [radius=.4pt];
\fill [opacity=1, green] (-1/3, 5/6) circle [radius=.4pt];
\fill [opacity=1, green] (5/6, -1/3) circle [radius=.4pt];
\fill [opacity=1, green] (-1/3, 4/3) circle [radius=.4pt];
\fill [opacity=1, green] (4/3, -1/3) circle [radius=.4pt];
\fill [opacity=1, green] (-1/6, 7/6) circle [radius=.4pt];
\fill [opacity=1, green] (7/6, -1/6) circle [radius=.4pt];
\fill [opacity=1, green] (1/6, 4/3) circle [radius=.4pt];
\fill [opacity=1, green] (4/3, 1/6) circle [radius=.4pt];
\fill [opacity=1, green] (1/3, 7/6) circle [radius=.4pt];
\fill [opacity=1, green] (7/6, 1/3) circle [radius=.4pt];
\fill [opacity=1, green] (2/3, 4/3) circle [radius=.4pt];
\fill [opacity=1, green] (4/3, 2/3) circle [radius=.4pt];
\fill [opacity=1, green] (5/6, 7/6) circle [radius=.4pt];
\fill [opacity=1, green] (7/6, 5/6) circle [radius=.4pt];
\fill [opacity=1, green] (7/6, 4/3) circle [radius=.4pt];
\fill [opacity=1, green] (4/3, 7/6) circle [radius=.4pt];
\fill [opacity=1, green] (4/3, 7/6) circle [radius=.4pt];
\fill [opacity=1, green] (7/6, 4/3) circle [radius=.4pt];
\draw [line width=0.8mm, dashed, cyan] (.28, .28) circle (7pt);
\draw [line width=.8mm, magenta ] (1/5, 1/6)--(1/3, 1/6);
\draw [line width=.8mm, magenta ] (1/3, 1/6)--(1/3, 1/5);
\fill [opacity=1, green] (1/3, 1/6) circle [radius=.4pt];
\end{tikzpicture}
\caption{$b[\calG(\del_3)]_1$ (red part), the small scale lattice $\frakL(\del_1, \del_2; 1)$ (green part), the line segment with length $\frac{1}{3}-\frac{1}{5}=\frac{2}{15}$ (the horizontal magenta colored line) and the line segment with length $\frac{1}{5}-\frac{1}{6}=\frac{1}{30}$ (the vertical magenta colored line).}
\label{Fig10}
\end{figure}
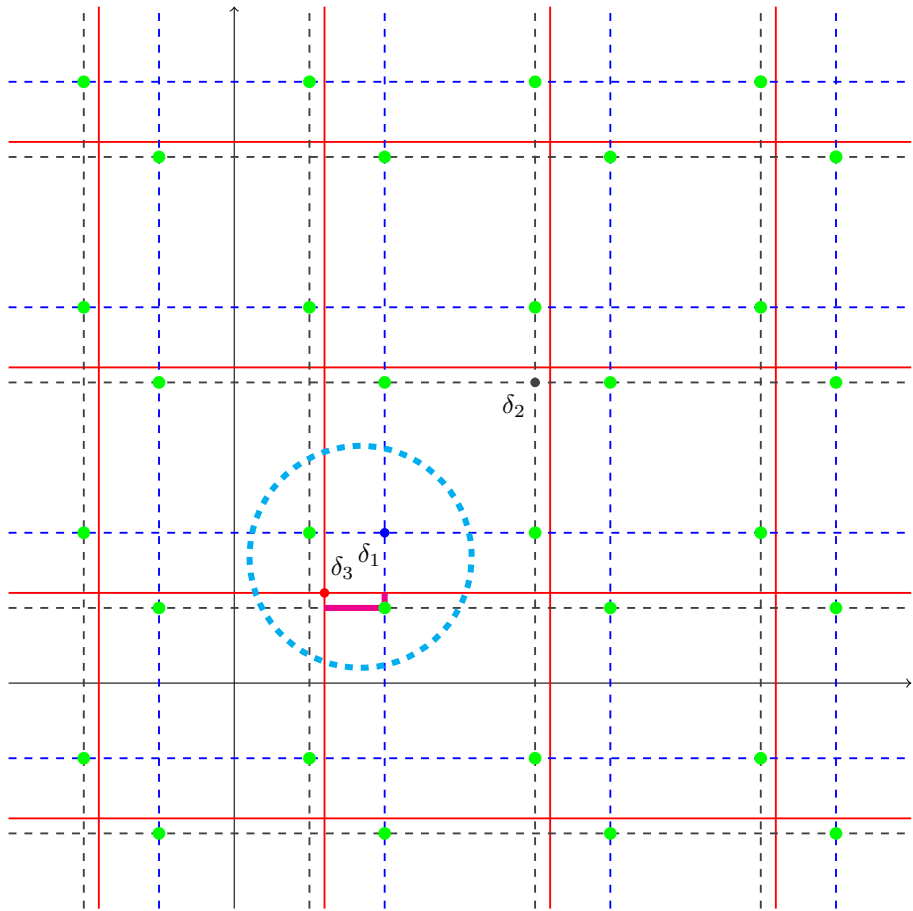
\end{center}

It turns out that this case $m=1$ already illustrates the main point.  The case $m=0$ is even easier, and for all other $m>0$ we will either get 
\begin{eqnarray*}
\textrm{dist} \left(b\left[\calG(\del_3)\right]_m,  \frakL(\del_1,\del_2; m) \right)%
&=& \inf_{k_1, k_2 \in \Z} \left|\left(\frac{1}{3} \pm \frac{k_1}{2^m} \right)- \left(\frac{1}{5} \pm \frac{k_2}{2^m} \right) \right| \\
&\geq& \inf_{k_1, k_2 \in \Z} \left|\frac{2}{15} - \frac{k}{2^m} \right| \ge  \frac{C}{2^m}
\end{eqnarray*}
by techniques in \cite[Proposition 2.4]{AHJOW} (see, also Lemma 3 in \cite{A}), or a similar claim where $\frac{1}{3}$ is replaced by $\frac{2}{3}$.  The other calculations for (a) and (b) are similar. For example, by an easy modification of the above arguments, we have for any $m \ge 0$, 
$$
\textrm{dist} \left(b\left[\calG(\del_1)\right]_m,  \frakL(\del_2,\del_3; m) \right)
$$
(this is for (a)) is either
$$
 \inf_{k_1, k_2 \in \Z} \left|\left(\frac{2}{3} \pm \frac{k_1}{2^m} \right)- \left(\frac{1}{3} \pm \frac{k_2}{2^m} \right) \right| 
 $$
 or
 $$
  \inf_{k_1, k_2 \in \Z} \left|\left(\frac{1}{5} \pm \frac{k_1}{2^m} \right)- \left(\frac{1}{3} \pm \frac{k_2}{2^m} \right) \right| 
$$
which is clearly bounded below by some $\frac{C}{2^m}$, where $C$ is independent of the choice of $m$.  We leave the details to the interested reader.

\medskip

We now verify Condition (ii).  We first calculate the limit infimum,  and begin by examining the case $j=1$.  In this case, it is clear from the Figure \ref{Fig11} that 
\[
\frac{\textrm{dist} \left( (m\calC)_{\bfVa_1}^{\del_1} (1), \calS_{\bfVa_2, \bfVa_{3}}^{\del_2, \del_{3}}(1) \right)}{2} =  \frac{1}{3}
\]
whereas
\[
\frac{\textrm{dist} \left( (m\calC)_{\bfVa_2}^{\del_2} (1), \calS_{\bfVa_1, \bfVa_{3}}^{\del_1, \del_{3}}(1) \right)}{2} =
 \frac{\textrm{dist} \left( (m\calC)_{\bfVa_3}^{\del_3} (1), \calS_{\bfVa_1, \bfVa_{2}}^{\del_1, \del_{2}}(1) \right)}{2} = \frac{|2/3-1/5|}{2} =\frac{7}{15}.
\]

\begin{center}
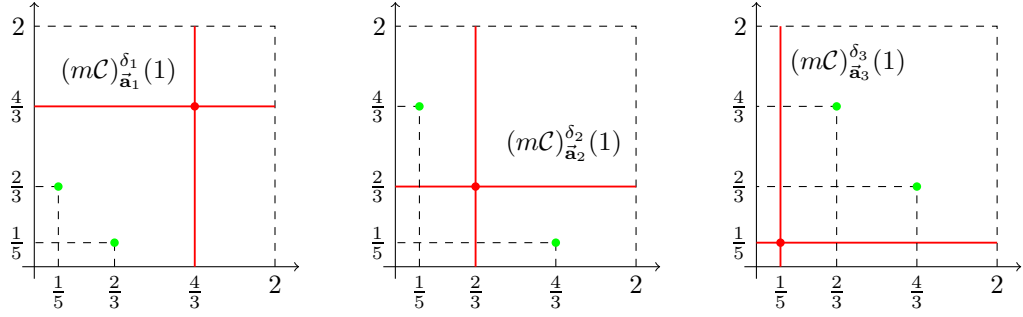
\begin{figure}[ht]
\begin{tikzpicture}[scale=1.6]
\draw (0, -.1) [->] -- (0, 2.2);
\draw (-.1, 0) [->] -- (2.2,0);
\draw [dashed] (2, 0)--(2, 2)--(0, 2); 
\fill [opacity=1, red] (4/3, 4/3) circle [radius=1pt];
\draw [line width=0.25mm, red ] (4/3, 0)--(4/3, 2);
\draw [line width=0.25mm, red ] (0, 4/3)--(2, 4/3); 
\fill (2, 0) node [below] {$2$};
\fill (0, 2) node [left]  {$2$}; 
\fill (4/3, 0) node [below] {$\frac{4}{3}$};
\fill (0, 4/3) node [left]  {$\frac{4}{3}$};
\draw [dashed] (2/3, 0)--(2/3, 1/5)--(0, 1/5); 
\draw [dashed] (1/5, 0)--(1/5, 2/3)--(0, 2/3);
\fill [opacity=1, green] (2/3, 1/5) circle [radius=1pt];
\fill [opacity=1, green] (1/5, 2/3) circle [radius=1pt];
\fill (2/3, 0) node [below] {$\frac{2}{3}$};
\fill (0, 2/3) node [left] {$\frac{2}{3}$};
\fill (1/5, 0) node [below] {$\frac{1}{5}$};
\fill  (0, 1/5) node [left] {$\frac{1}{5}$};
\draw (3, -.1) [->] -- (3, 2.2);
\draw (2.9, 0) [->] -- (5.2,0);
\draw [dashed] (5, 0)--(5, 2)--(3, 2); 
\fill (5, 0) node [below] {$2$};
\fill (3, 2) node [left]  {$2$}; 
\fill [opacity=1, red] (11/3, 2/3) circle [radius=1pt];
\draw [line width=0.25mm, red ] (3, 2/3)--(5, 2/3);
\draw [line width=0.25mm, red ] (11/3, 0)--(11/3, 2);
\fill (11/3, 0) node [below] {$\frac{2}{3}$};
\fill (3, 2/3) node [left] {$\frac{2}{3}$};
\draw [dashed] (13/3, 0)--(13/3, 1/5)--(3, 1/5);
\draw [dashed] (16/5, 0)--(16/5, 4/3)--(3, 4/3);
\fill [opacity=1, green] (13/3, 1/5) circle [radius=1pt];
\fill [opacity=1, green] (16/5, 4/3) circle [radius=1pt];
\fill (13/3, 0) node [below] {$\frac{4}{3}$};
\fill (16/5, 0) node [below] {$\frac{1}{5}$};
\fill (3, 1/5) node [left] {$\frac{1}{5}$};
\fill (3, 4/3) node [left] {$\frac{4}{3}$};
\draw (6, -.1) [->] -- (6, 2.2);
\draw (5.9, 0) [->] -- (8.2,0);
\draw [dashed] (8, 0)--(8, 2)--(6, 2); 
\fill (8, 0) node [below] {$2$};
\fill (6, 2) node [left]  {$2$}; 
\fill [opacity=1, red] (6.2, .2) circle [radius=1pt];
\draw [line width=0.25mm, red ] (6, 1/5)--(8, 1/5);
\draw [line width=0.25mm, red ] (6.2, 0)--(6.2, 2);
\fill (6.2, 0) node [below] {$\frac{1}{5}$};
\fill (6, .2) node [left] {$\frac{1}{5}$};
\draw [dashed] (22/3, 0)--(22/3, 2/3)--(6, 2/3);
\draw [dashed] (20/3, 0)--(20/3, 4/3)--(6, 4/3);
\fill [opacity=1, green] (22/3, 2/3) circle [radius=1pt];
\fill [opacity=1, green] (20/3, 4/3) circle [radius=1pt];
\fill (22/3, 0) node [below] {$\frac{4}{3}$};
\fill (20/3, 0) node [below] {$\frac{2}{3}$};
\fill (6, 2/3) node [left] {$\frac{2}{3}$};
\fill (6, 4/3) node [left] {$\frac{4}{3}$};
\fill (.7, 1.4) node [above] {$(m\calC)_{\bfVa_1}^{\del_1}(1)$};
\fill (4.4, .8) node [above] {$(m\calC)_{\bfVa_2}^{\del_2}(1)$};
\fill (6.2, 1.7) node [right] {$(m\calC)_{\bfVa_3}^{\del_3}(1)$};
\end{tikzpicture}
\caption{All modulated corner sets $(m\calC)_{\bfVa_i}^{\del_i}(1)$ for $i=1, 2, 3$ (red parts) and all corresponding large scale samplings $\calS_{\bfVa_1, \dots, \widehat{\bfVa_i}, \dots, \bfVa_3}^{\del_1, \dots, \widehat{\del_i}, \dots, \del_3}(1)$ for $i=1, 2, 3$ (green parts).}
\label{Fig11}
\end{figure}
\end{center}

Let us also calculate the case when $j=2$ (see, Figure \ref{Fig12}). In this case, we have
\[
\frac{\textrm{dist} \left( (m\calC)_{\bfVa_2}^{\del_2} (2), \calS_{\bfVa_1, \bfVa_{3}}^{\del_2, \del_{3}}(2) \right)}{4} =  \frac{1}{3}
\]
whereas
\[
\frac{\textrm{dist} \left( (m\calC)_{\bfVa_1}^{\del_1} (2), \calS_{\bfVa_2, \bfVa_{3}}^{\del_2, \del_{3}}(2) \right)}{4} =
 \frac{\textrm{dist} \left( (m\calC)_{\bfVa_3}^{\del_3} (2), \calS_{\bfVa_1, \bfVa_{2}}^{\del_1, \del_{2}}(2) \right)}{4} = \frac{4/3-1/5}{2} =\frac{17}{60}.
\]

\begin{center}
\begin{figure}[ht]
\begin{tikzpicture}[scale=.8]
\draw (0, -.1) [->] -- (0, 4.4);
\draw (-.1, 0) [->] -- (4.4,0);
\draw [dashed] (4, 0)--(4, 4)--(0, 4); 
\fill [opacity=1, red] (4/3, 4/3) circle [radius=2pt];
\draw [line width=0.25mm, red ] (4/3, 0)--(4/3, 4);
\draw [line width=0.25mm, red ] (0, 4/3)--(4, 4/3); 
\fill (4, 0) node [below] {$4$};
\fill (0, 4) node [left]  {$4$}; 
\fill (4/3, 0) node [below] {$\frac{4}{3}$};
\fill (0, 4/3) node [left]  {$\frac{4}{3}$};
\draw [dashed] (8/3, 0)--(8/3, 1/5)--(0, 1/5); 
\draw [dashed] (1/5, 0)--(1/5, 8/3)--(0, 8/3);
\fill [opacity=1, green] (8/3, 1/5) circle [radius=2pt];
\fill [opacity=1, green] (1/5, 8/3) circle [radius=2pt];
\fill (8/3, 0) node [below] {$\frac{8}{3}$};
\fill (0, 8/3) node [left] {$\frac{8}{3}$};
\fill (1/5, 0) node [below] {$\frac{1}{5}$};
\fill  (0, 1/5) node [left] {$\frac{1}{5}$};
\draw (6, -.1) [->] -- (6, 4.4);
\draw (5.8, 0) [->] -- (10.4,0);
\draw [dashed] (10, 0)--(10, 4)--(6, 4); 
\fill (10, 0) node [below] {$4$};
\fill (6, 4) node [left]  {$4$}; 
\fill [opacity=1, red] (26/3, 8/3) circle [radius=2pt];
\draw [line width=0.25mm, red ] (6, 8/3)--(10, 8/3);
\draw [line width=0.25mm, red ] (26/3, 0)--(26/3, 4);
\fill (26/3, 0) node [below] {$\frac{8}{3}$};
\fill (6, 8/3) node [left] {$\frac{8}{3}$};
\draw [dashed] (22/3, 0)--(22/3, 1/5)--(6, 1/5);
\draw [dashed] (31/5, 0)--(31/5, 4/3)--(6, 4/3);
\fill [opacity=1, green] (22/3, 1/5) circle [radius=2pt];
\fill [opacity=1, green] (31/5, 4/3) circle [radius=2pt];
\fill (22/3, 0) node [below] {$\frac{4}{3}$};
\fill (31/5, 0) node [below] {$\frac{1}{5}$};
\fill (6, 1/5) node [left] {$\frac{1}{5}$};
\fill (6, 4/3) node [left] {$\frac{4}{3}$};
\draw (12, -.1) [->] -- (12, 4.4);
\draw (11.8, 0) [->] -- (16.4,0);
\draw [dashed] (16, 0)--(16, 4)--(12, 4); 
\fill (16, 0) node [below] {$4$};
\fill (12, 4) node [left]  {$4$}; 
\fill [opacity=1, red] (12.2, .2) circle [radius=2pt];
\draw [line width=0.25mm, red ] (12, 1/5)--(16, 1/5);
\draw [line width=0.25mm, red ] (12.2, 0)--(12.2, 4);
\fill (12.2, 0) node [below] {$\frac{1}{5}$};
\fill (12, .2) node [left] {$\frac{1}{5}$};
\draw [dashed] (40/3, 0)--(40/3, 8/3)--(12, 8/3);
\draw [dashed] (44/3, 0)--(44/3, 4/3)--(12, 4/3);
\fill [opacity=1, green] (40/3, 8/3) circle [radius=2pt];
\fill [opacity=1, green] (44/3, 4/3) circle [radius=2pt];
\fill (44/3, 0) node [below] {$\frac{8}{3}$};
\fill (40/3, 0) node [below] {$\frac{4}{3}$};
\fill (12, 4/3) node [left] {$\frac{4}{3}$};
\fill (12, 8/3) node [left] {$\frac{8}{3}$};
\fill (2.4, 2.8) node [above] {$(m\calC)_{\bfVa_1}^{\del_1}(2)$};
\fill (7.5, 1.6) node [above] {$(m\calC)_{\bfVa_2}^{\del_2}(2)$};
\fill (12.3, 3.4) node [right] {$(m\calC)_{\bfVa_3}^{\del_3}(2)$};
\end{tikzpicture}
\caption{All modulated corner sets $(m\calC)_{\bfVa_i}^{\del_i}(2)$ for $i=1, 2, 3$ (red parts) and all corresponding large scale samplings $\calS_{\bfVa_1, \dots, \widehat{\bfVa_i}, \dots, \bfVa_3}^{\del_1, \dots, \widehat{\del_i}, \dots, \del_3}(2)$ for $i=1, 2, 3$ (green parts).}
\label{Fig12}
\end{figure}
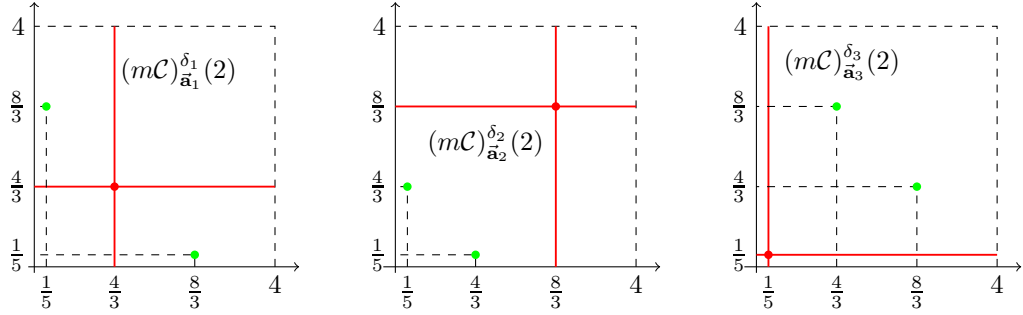
\end{center}

Similarly, the cases $j=1, 2$ (that is, $m=-1$ and $-2$, respectively) contain all the main ideas for the large scale case (that is, $m<0$). Recall for $j \ge 0$, 
$$
\del_1+\calL_{\bfVa_1}(j)= \begin{cases} 
\left(\frac{2^{j}}{3}, \frac{2^{j}}{3} \right),  & \hfill j \ \textrm{even} \\
\\
\left(\frac{2^{j+1}}{3}, \frac{2^{j+1}}{3} \right),  & \hfill j \ \textrm{odd}
\end{cases}, 
\quad 
\del_2+\calL_{\bfVa_2}(j)= \begin{cases} 
\left(\frac{2^{j+1}}{3}, \frac{2^{j+1}}{3} \right), & \hfill j  \ \textrm{even} \\
\\
\left(\frac{2^{j}}{3}, \frac{2^{j}}{3} \right), & \hfill j \ \textrm{odd}
\end{cases}, 
$$
and 
$$
\del_3+\calL_{\bfVa_3}(j) \equiv \left(\frac{1}{5}, \frac{1}{5} \right).
$$

This suggests that the first two limit infimums, which are, 
$$
\liminf_{j \to \infty} \frac{\textrm{dist} \left( (m\calC)_{\bfVa_1}^{\del_1} (j), \calS_{\bfVa_2, \bfVa_{3}}^{\del_2, \del_{3}}(j) \right)}{2^j} \quad \textrm{and} \quad  \liminf_{j \to \infty} \frac{\textrm{dist} \left( (m\calC)_{\bfVa_2}^{\del_2} (j), \calS_{\bfVa_1, \bfVa_{3}}^{\del_1, \del_{3}}(j) \right)}{2^j},
$$
are either 
$$
\lim_{j \to \infty} \frac{2^{j+1}/3-2^j/3}{2^j}=\frac{1}{3}
$$
or 
$$
\lim_{j \to \infty} \frac{2^j/3-1/5}{2^j}=\frac{1}{3}.
$$
To see this, one may consider two different cases by requiring $j$ being even or being odd, the desired claim will then follow from plotting the corresponding modulated corner sets and large scale samplings out, as the first two graphs illustrated in Figure \ref{Fig11}. 

While for the third limit infimum, that is, 
$$
 \liminf_{j \to \infty} \frac{\textrm{dist} \left( (m\calC)_{\bfVa_3}^{\del_3} (j), \calS_{\bfVa_1, \bfVa_{2}}^{\del_1, \del_{2}}(j) \right)}{2^j},
$$
we can see that it will be
$$
\lim_{j \to \infty} \frac{2^j/3-1/5}{2^j}=\frac{1}{3}.
$$
One may consult the third graph in Figure \ref{Fig11} for a visualization of this case. Overall, we can take $\frac{1}{3}$ as the limit infimum in Condition (ii) of our main result.

The calculations for the limit supreme are similar and can be visualized from Figure \ref{Fig11}. More precisely, we can see that the quantity
$$
\max_{i=1,2,3} \left[\limsup_{j \to \infty} \frac{\max_{y' \in \calS_{\bfVa_1,  \dots, \widehat{\bfVa_i}, \dots, \bfVa_3}^{\del_1, \dots, \widehat{\del_i}, \dots, \del_3}(j)} \textrm{dev}\left( (m\calC)_{\bfVa_i}^{\del_i} (j), y' \right)}{2^j}\right]
$$
is the limit
$$
\lim_{j \to \infty} \frac{2^{j+1}/3 - 1/5}{2^j} = \frac{2}{3}.
$$
This verifies Condition (ii) in Theorem \ref{mainresult}.  \hfill $\square$

\end{document}